\documentclass[11pt,centertags,oneside]{article}
\usepackage{amsmath,amstext,amsthm,amscd,typearea}
\usepackage{amssymb}
\usepackage{a4wide}
\usepackage[mathscr]{eucal}
\usepackage{mathrsfs}
\usepackage{typearea}
\usepackage{charter}
\usepackage{pdfsync}
\usepackage{url}
\usepackage[T1]{fontenc}

\usepackage{hyperref}

\usepackage{xcolor}

\usepackage[a4paper,width=15.2cm,top=4cm,bottom=4cm]{geometry}

\numberwithin{equation}{section}

\newtheorem{theorem}{Theorem}[section]

\newtheorem{proposition}[theorem]{Proposition}

\newtheorem{lemma}[theorem]{Lemma}
\newtheorem{remark}[theorem]{Remark}

\newtheorem{problem}[theorem]{Problem}

\newcommand{\cali}[1]{\mathscr{#1}}

\newcommand{\vol}{\mathop{\mathrm{vol}}}

\newcommand{\ddc}{dd^c}
\newcommand{\dc}{d^c}

\newcommand{\PSH}{{\rm PSH}}

\newcommand{\id}{{\rm id}}

\newcommand{\Cc}{\cali{C}}

\newcommand{\capK}{\text{cap}}

\newcommand{\B}{\mathbb{B}}
\newcommand{\C}{\mathbb{C}}

\newcommand{\N}{\mathbb{N}}

\newcommand{\R}{\mathbb{R}}

\title{\bf Quantitative stability for the complex Monge-Amp\`ere equations I}
\providecommand{\keywords}[1]{\textbf{\textit{Keywords:}} #1}
\providecommand{\subject}[1]{\textbf{\textit{Mathematics Subject Classification 2010:}} #1}

\author{Hoang-Son Do and Duc-Viet Vu}

\newcommand{\Addresses}{{
		\bigskip
		\footnotesize
		\textsc{Duc-Viet Vu, University of Cologne, Division of Mathematics, Department of Mathematics and Computer Science, Weyertal 86-90, 50931, K\"oln,  Germany}
		\noindent
		\par\nopagebreak
		\noindent
		\textit{E-mail address}: \texttt{vuviet@math.uni-koeln.de}	
		
	\bigskip
\footnotesize
\textsc{Hoang-Son Do, Vietnam Academy of Science and Technology, Institute of Mathematics, 18 Hoang Quoc Viet road, Cau Giay, Hanoi, Vietnam}
\noindent
\par\nopagebreak
\noindent
\textit{E-mail address}: \texttt{dhson@math.ac.vn}	}}
\date{\today}
\begin{document}
\maketitle
\begin{abstract} We generalize several known stability estimates for complex Monge-Amp\`ere equations to the setting of low (or high) energy potentials. We apply our estimates to obtain, among other things, a  quantitative domination principle, and metric properties of the space of potentials of finite energy. Further applications will be given in subsequent papers.   
\end{abstract}
\noindent
, {convex weights}, {lower energy}, {non-pluripolar products}.
\\

\noindent
\subject{32U15}, {32Q15}.


\section{Introduction}

Let $(X,\omega)$ be a compact K\"ahler manifold of dimension $n$ and let $\alpha$ be a big cohomology $(1,1)$-class in $X$. Let $\theta$ be a closed smooth real $(1,1)$-form in $\alpha$. For $u \in \PSH(X, \theta)$, we put $\theta_u:= \ddc u+ \theta$. 
Let $\phi \in \PSH(X, \theta)$ such that $\phi \le 0$ and $\int_X \theta_\phi^n >0$, where $\theta_\phi^n$ denotes the non-pluripolar self-product of $\theta_\phi$ (see \cite{BEGZ,BT_fine_87}). Denote by $\PSH(X,\theta, \phi)$ the set of $\theta$-psh functions $u$ with $u \le \phi$. Note that it is slightly different from the usual definition of $\PSH(X,\theta,\phi)$ in which  $u$ is only required to be more singular than $\phi$. This difference is not essential.  We say that $\phi$ is a \emph{model $\theta$-psh function} (see \cite{Lu-Darvas-DiNezza-mono,Ross-WittNystrom}) if $\phi= P_\theta[\phi]$ and $\int_X \theta_\phi^n >0$, where 
$$P_\theta[\phi]:= \big(\sup \{ \psi \in \PSH(X,\theta): \psi \le 0, \, \psi \le \phi+ O(1)\}\big)^*.$$
The function $P_\theta[\phi]$ is called a roof-top envelope in \cite{Lu-Darvas-DiNezza-mono}.  By \cite{Lu-Darvas-DiNezza-mono}, the function $P_\theta[u]$ is a model one for every $u \in \PSH(X,\theta)$ with $\int_X \theta_u^n >0$, and for every $u \in \PSH(X,\theta,\phi)$ with $\int_X \theta_u^n= \int_X \theta_\phi^n$ we have $P_\theta[u]= P_\theta[\phi]$.

Let $\phi$ be now a model $\theta$-psh function. Let $\mathcal{E}(X, \theta,\phi)$ be the space of $\theta$-psh functions $u \le \phi$ with $\int_X \theta_u^n = \int_X \theta_\phi^n$.  Let $\mu$ be a non-pluripolar measure with $\mu(X)= \int_X \theta_\phi^n$. It was proved in \cite{Lu-Darvas-DiNezza-logconcave} (see also \cite{Vu_Do-MA,Lu-Darvas-DiNezza-mono}) that the Monge-Amp\`ere equation with prescribed singularities 
 \begin{align}\label{eq-MAphi}
(\ddc u+ \theta)^n  = \mu, \quad u \in \PSH(X, \theta,\phi)
\end{align}
admits a unique solution $u \in \mathcal{E}(X,\theta, \phi)$ and $\sup_X(u-\phi)=0$.  We note that the left-hand side of (\ref{eq-MAphi}) denotes the non-pluripolar self-product of $\theta_u$ (see \cite{BT_fine_87,BEGZ,GZ-weighted,Viet-generalized-nonpluri}). We refer to \cite{BEGZ,Cegrell,Dinew-uniqueness,Kolodziej_Acta,Yau1978}, to cite a few, for the well-known case where $\alpha$ is big and $\phi$ is a potential of minimal singularities in $\alpha$.

The aim of this paper is to study the following stability question for the equation (\ref{eq-MAphi}).

\begin{problem} \label{pro-stability} Let $\theta, \phi$ be as above. Let $u_j \in \mathcal{E}(X, \theta, \phi)$ for $j=1,2$ and $\mu_j:= \theta_{u_j}^n$ for $j=1,2$.  Compare $u_1$ with $u_2$ in terms of a suitable ``distance'' between $\mu_1,\mu_2$? 
\end{problem}

To our best knowledge, there has been no available quantitative comparison between potentials of finite energy in general, even in the case where $\alpha$ is K\"ahler and $\phi\equiv 0$. The closest result that we know of is the uniqueness property (by \cite{Dinew-uniqueness} in the K\"ahler case and by \cite{BEGZ,Lu-Darvas-DiNezza-logconcave} in the present setting) which says that $u_1= u_2$ if $\mu_1 =\mu_2$. There were however some concrete estimates for the distance between $u_1,u_2$ in terms of $\mu_1,\mu_2$ but one had to assume some extra assumption (i.e., $u_1,u_2 \in \mathcal{E}^1(X,\theta,\phi)$); see \cite{Blocki_stability,Guedj-Zeriahi-big-stability}. We will explain details below.

The goal of this paper  is to solve Problem \ref{pro-stability} for any potential of high or low energy. As one will see in our applications later in this paper or in our subsequent paper, it is crucial to consider Problem \ref{pro-stability} for potentials in low energy.

Let $\widetilde{\mathcal{W}}^-$ be the set of convex, non-decreasing functions $\chi: \R_{\le 0} \to \R_{ \le 0 }$ such that $\chi(0)=0$ and $\chi \not \equiv 0$. Let $\mathcal{W}^-$ be the subset of $\chi \in \widetilde{\mathcal{W}}^-$ such that $\chi(-\infty)= -\infty$.  Note that in general $\chi \in \widetilde{\mathcal{W}}^-$ can be bounded. 
It is crucial in our method that we consider also bounded weights $\chi \in \widetilde{\mathcal{W}}^-$.  Let $M \ge 1$ be a constant and  $\mathcal{W}^+_M$ the usual space of  increasing concave functions $\chi: \R_{\le 0} \to \R_{ \le 0}$ such that $\chi(0)=0$, $\chi<0$ on $(-\infty, 0)$, and $|t \chi'(t)| \le M|\chi(t)|$ for every $t \le 0$.

Let $\varrho:= \int_X \theta_\phi^n$. For  $\chi \in \widetilde{\mathcal{W}}^- \cup  \mathcal{W}^+_M$ and   $u \in \PSH(X,\theta,\phi)$, let 
 $$E^0_{\chi, \theta, \phi}(u):= - \varrho^{-1} \int_X \chi(u- \phi) \theta_u^n$$
 which is called \emph{the (normalized) $\chi$-energy} of $u$ (with respect to $\theta, \phi$).   We denote
$$\mathcal{E}_\chi(X, \theta,\phi):=\big\{u \in \mathcal{E} (X, \theta, \phi): E^0_{\chi, \theta, \phi}(u)<\infty  \big\}.$$ Certainly if $\chi$ is bounded, then $\mathcal{E}_\chi(X, \theta,\phi)= \mathcal{E}(X, \theta,\phi)$. We would like to point out however that our method is not about the finiteness of $E^0_{\chi, \theta, \phi}(u)$ but estimating the size of that quantity. Thus whether $\chi$ is bounded or not does not make much difference for our later arguments.  
  Put
$$\quad I^0_\chi(u,v):=  \varrho^{-1}\int_{\{u<v\}} \chi(u-v) (\theta_v^n - \theta_u^n)+\varrho^{-1}\int_{\{u>v\}} \chi(v-u) (\theta_u^n - \theta_v^n)$$
 for $u,v \in \mathcal{E}_{\chi}(X, \theta, \phi)$. The factor $\varrho^{-1}$ in the defining formulae for $E^0_{\chi,\theta, \phi}(u)$ and  $I^0_\chi(u,v)$ plays the role of a normalizing constant. In geometric applications it is important to treat the case where $\varrho \to 0$, \emph{i.e,} to obtain estimates uniformly  as $\varrho \to 0$. 
 
Clearly if $\theta_u^n= \theta_v^n$, then $I^0_\chi(u,v)=0$. We will see later that each term in the sum defining $I^0_\chi(u,v)$ is nonnegative. 
We recall that there is a natural (quasi-)metric on the space $\mathcal{E}_\chi(X,\theta, \phi)$ constructed in \cite{Darvas_book,Darvas-lower-energy, Gupta}, and see \cite{DDL-L1metric,Lu-DiNezza-Lpmetric,Trusiani-energy,Xia-energy} as well.  The functional $I^0_\chi(u,v)$ has an intimate relation with these quasi-metrics. We refer to the end of Section \ref{sec-fixedtype} for details on this connection.  Here is  our main result.

\begin{theorem} \label{th-lowerenergy-phanintro}  Let $\theta$ be a closed smooth real $(1,1)$-form  and $\phi$ be a  negative $\theta$-psh function such that $\varrho:= \int_X \theta_\phi^n>0$.  Let $\chi, \tilde{\chi} \in \widetilde{\mathcal{W}}^-\cup\mathcal{W}^+_M$ ($M\geq 1$) such that $\tilde{\chi} \le \chi$, and if $\chi \in \widetilde{\mathcal{W}}^-$, then $\lim_{t\to -\infty} \frac{\chi(t)}{\tilde{\chi}(t)}=0$. Let $B \ge 1$ be a constant and let $u_j, \psi_j \in \mathcal{E}(X, \theta, \phi)$ satisfy $u_1 \le u_2$ and 
$$E^0_{\tilde{\chi}, \theta, \phi}(u_j)+E^0_{\tilde{\chi}, \theta, \phi}(\psi_j) \le B,$$
for $j=1,2$. Then there exist a constant $C>0$ depending only on $n, \tilde{\chi}(-1)$ and $M$,
 and a continuous increasing function $f: \R_{\ge 0} \to \R_{ \ge 0}$ depending only on $\chi, \tilde{\chi}$ such that $f(0)=0$ and 
\begin{align} \label{ine-maintheorem1short}
\int_X  -\chi (u_1-u_2)  (\theta_{\psi_1}^n- \theta_{\psi_2}^n) \le 
C\varrho B^2 f^{\circ n}\big(I^0_\chi(u_1,u_2)\big),
\end{align}
where $f^{\circ n}:= f \circ f \circ \cdots \circ f$ ($n$-iterate of $f$). Moreover, if 
$\phi=P_{\theta}[\phi]$ and $\sup_Xu_1=\sup_Xu_2$ then 
\begin{align} \label{ine2-maintheorem1short}
	\int_X  -\chi (u_1-u_2)  (\theta_{\psi_1}^n+\theta_{\psi_2}^n) \le 
	\varrho\, g\big(I^0_\chi(u_1,u_2)\big),
\end{align}
where $g: \R_{\ge 0} \to \R_{ \ge 0}$ is a continuous increasing function depending only on $n, M, X, \omega, \theta,\chi, \tilde{\chi}$ and $B$ such that $g(0)=0$.
\end{theorem}

If $\chi \in \mathcal{W}^+_M$, then one can certainly apply Theorem \ref{th-lowerenergy-phanintro} to $\tilde{\chi}=\chi$. Nevertheless, we underline that in applications it is of crucial importance to consider $\chi \in \widetilde{\mathcal{W}}^-$. In this case in order to have (\ref{ine2-maintheorem1short}), it is necessary to require an upper bound for $\tilde{\chi}$-energy of $u_j$, where $\tilde{\chi}$  ``dominates'' $\chi$ as in the statement of Theorem \ref{th-lowerenergy-phanintro}. We refer to Subsection \ref{ex-chingabangchi} for details. 

One sees that (\ref{ine2-maintheorem1short}) implies, in particular, that if $I^0_\chi(u_1,u_2) \to 0$, then the expression in the left-hand side also converges to $0$. 
Theorem \ref{th-lowerenergy-phanintro} follows from Theorems \ref{th-lowerenergy} and
\ref{th-lowerenergy-kocochuanhoa}  below, where the functions $f$ and $g$ are given explicitly.  
We note that \emph{the single theorem \ref{th-lowerenergy-phanintro} contains the following three important results in pluripotential theory: uniqueness of solutions of complex Monge-Amp\`ere equations, domination principle, and comparison of capacities}. We obtain indeed quantitative (hence stronger) versions of these results for which we refer to Section \ref{sec-appli}. The quantitative version of uniqueness theorem (see Theorem \ref{thequantitativeuniqueness2} below) provides an answer to Problem \ref{pro-stability}. Readers also find, in Section \ref{sec-appli}, a quantitative version of the fact that the convergence in Darvas's metric in $\mathcal{E}_\chi(X, \theta, \phi)$ implies the convergence in capacity. Notice that such an estimate seems to be not reachable by using the usual plurisubharmonic envelope method.

The main novelty of Theorem \ref{th-lowerenergy-phanintro} is that it deals with \emph{arbitrary} weights. Similar statements was already known for $\chi(t)=t$ (see \cite{Blocki_stability,BBEGZ-crelle,Guedj-Zeriahi-big-stability,Trusiani-strong-topo}). However the proof there only work \emph{exclusively} for this case. One should notice that the weight $\chi(t)=t$ is very special: it is linear and lies in the middle between higher energy weights and lower energy weights. As to the proof of Theorem \ref{th-lowerenergy-phanintro},  going up to the space of higher energy weights or going down to the space of lower energy weights are equally difficult.   We will explain this point in more details in the paragraph after Theorem \ref{the-mainstabilitylowenergyconvexintro} below.

The key in the proof of Theorem \ref{th-lowerenergy-phanintro} is Proposition \ref{pro-mainstabilitylowenergyconvex} in Section 3 a simplified version of which we state here for readers' convenience.

\begin{theorem} \label{the-mainstabilitylowenergyconvexintro} Let $\chi, \tilde{\chi} \in\widetilde{\mathcal{W}}^-\cup\mathcal{W}^{+}_M$	such that $\tilde{\chi}\leq\chi$ and $\chi\in \Cc^1(\R)$.  Let $u_1, u_2, u_3 \in \mathcal{E}(X, \theta, \phi)$ such that $u_1 \le u_2$ and
	$u_j-\phi$ is bounded ($j=1, 2, 3$), where $\phi$ is a negative $\theta$-psh function satisfying 
	$\varrho:=\vol(\theta_\phi)>0$.  Then there exist a constant $C>0$ depending only on $n, \tilde{\chi}(-1)$ 
	and $M$  such that 
	\begin{align*} 
		\int_X  \chi'(u_1-u_2) d(u_1- u_2) \wedge \dc (u_1- u_2) \wedge \theta_{u_3}^{n-1}  \le C \varrho B^2 f^{\circ (n-1)}\big(I^0_\chi(u_1,u_2)\big),		
	\end{align*}
	where $B:=\sum_{j=1}^3 \max\{E^0_{\tilde{\chi},\theta, \phi}(u_j),1 \}$ and
	and  $f: \R_{\ge 0} \to \R_{\ge 0}$  is a continuous function such that $f(0)=0$ if one has either $\chi \in \mathcal{W}^{+}_M$ or $\chi \in \widetilde{\mathcal{W}}^{-}$ and $\lim_{t\to -\infty} \frac{\chi(t)}{\tilde{\chi}(t)}=0$.
\end{theorem}

As far as we know, all of previous works related to Theorem \ref{the-mainstabilitylowenergyconvexintro} only concern with $\chi(t)=t$.  In this case, Theorem \ref{the-mainstabilitylowenergyconvexintro} was known with an explicit $f$ and without $\tilde{\chi}$ if   $\phi$ is of minimal singularity in the cohomology class of $\theta$, by \cite{Blocki_stability,Guedj-Zeriahi-big-stability}. 

The key ingredients in previous versions of  Theorem \ref{the-mainstabilitylowenergyconvexintro} for $\chi(t)=t$ are integration by parts arguments. Direct generalization of such reasoning immediately break down if $\chi  \not= \id$:  in  a more precise but technical level,  the integration by parts arguments give terms like $\chi'(u_1- u_2)d (u_1-u_3) \wedge \dc (u_1- u_3)$, such quantity is easy to bound if $\chi = \id$ (hence $\chi' \equiv 1$), but it is no longer the case if $\chi \not = \id$.  

In order to prove Theorem \ref{the-mainstabilitylowenergyconvexintro}, we still use this strategy but need to use a so-called ``monotonicity argument" from \cite{Vu_Do-MA,Viet-generalized-nonpluri,Viet-convexity-weightedclass} to deal with  general $\chi$. In a nutshell it is about using intensively the pluri-locality of Monge-Amp\`ere operators together with the monotonicity of pluricomplex energy which allows one to bound from above ``Monge-Amp\`ere quantities'' of bad potentials by that of nicer potentials.  This method is a flexible tool to deal with ``low regularity'', and was a key in the proof of the convexity of the class of potentials of finite $\chi$-energy  in \cite{Viet-convexity-weightedclass}, as well as, giving a characterization of the class of Monge-Amp\`ere measures with potentials of finite $\chi$-energy in \cite{Vu_Do-MA}.

We refer to the end of the paper for some applications of our main results. Furthermore, the quantitative domination principle obtained in Section \ref{sec-appli} was used crucially in \cite{DangVu} to describe the degeneration of conic K\"ahler-Einstein metrics. 
We note also that the present paper is the first part of the manuscript  \cite{Vu_DoHS-quantitative}, in which we give a more or less satisfactory treatment for a much more general question than Problem \ref{pro-stability}: precisely, we establish quantitative stability when both the cohomology class and the singularity type vary. The second part of \cite{Vu_DoHS-quantitative}, where this generalization is treated, will be submitted separately due to the length constraint.

The paper is organized as follows. In Section \ref{sec-preli}, we recall the integration by parts formula from \cite{Viet-convexity-weightedclass}, auxiliary facts about weights are also collected there. Theorems \ref{th-lowerenergy-phanintro} and \ref{the-mainstabilitylowenergyconvexintro} are proved  in Section \ref{sec-fixedtype}. Applications will be given in Section \ref{sec-appli}.\\

\noindent
\textbf{Acknowledgements.} We would like to thank Tam\'as Darvas, Vincent Guedj, Henri Guenancia,  Prakhar Gupta,  Hoang Chinh Lu, Tat Dat T\^o, Valentino Tosatti, and Ahmed Zeriahi for fruitful discussions. We also want to express our gratitude to anonymous referees for their careful reading and suggestions.  The research of D.-V. Vu is partly funded by the Deutsche Forschungsgemeinschaft (DFG, German Research Foundation)-Projektnummer 500055552.
The research of H.-S. Do  is  funded by International Centre for Research and Postgraduate Training in Mathematics (ICRTM) under grant number ICRTM01\_2023.01.
  \\

\section{Preliminaries} \label{sec-preli}

\subsection{Integration by parts} \label{subsec-intebyparts}

In this subsection, we recall the integration by parts formula obtained in \cite[Theorem 2.6]{Viet-convexity-weightedclass}. This formula will play a key role in our proof of main results later.

Let $X$ be a compact K\"ahler manifold.  Let $T_1, \ldots, T_m$ be closed positive $(1,1)$-currents on $X$. Let $T$ be  a closed positive current of bi-degree $(p,p)$ on $X$. The \emph{$T$-relative non-pluripolar product} $\langle \wedge_{j=1}^m T_j \dot{\wedge} T\rangle$ is defined  in a way similar to that of  the usual non-pluripolar product (see \cite{Viet-generalized-nonpluri}). The product $ \langle  \wedge_{j=1}^m T_j \dot{\wedge} T\rangle $ is a closed positive current of bi-degree $(m+p,m+p)$; and the wedge product $ \langle  \wedge_{j=1}^m T_j \dot{\wedge} T\rangle $ as an operator on currents  is  symmetric with respect to $T_1, \ldots, T_m$ and is homogeneous. In latter applications, we will only use the case where $T$ is the non-pluripolar product of some closed positive $(1,1)$-currents, say, $T= \langle T_{m+1} \wedge \cdots \wedge T_{m+l} \rangle$, where $T_j$ is $(1,1)$-currents for $m+1 \le j \le m+l$. In this case, $\langle T_1 \wedge \cdots \wedge T_m \dot{\wedge}T \rangle $ is simply equal to $\langle \wedge_{j=1}^{m+l} T_j \rangle$. We usually remove the bracket $\langle \quad \rangle$ in the non-pluripolar product to ease the notation. 

Recall that a \emph{dsh} function on $X$ is the difference of two quasi-plurisubharmonic (quasi-psh for short) functions on $X$ (see \cite{DS_tm}). These functions are well-defined outside pluripolar sets. Let $v$ be a dsh function on $X$.  Let $T$ be a closed positive current on $X$. We say that $v$ is \emph{$T$-admissible} if  there exist  quasi-psh functions $\varphi_1, \varphi_2$ such that $v= \varphi_1- \varphi_2$  and $T$ has no mass on $\{\varphi_j=-\infty\}$ 
for $j=1,2$. In particular, if $T$ has no mass on pluripolar sets, then every dsh function is $T$-admissible.  

Assume now that $v$ is $T$-admissible.    Let $\varphi_{1}, \varphi_{2}$ be quasi-psh functions such that $v= \varphi_{1}- \varphi_{2}$ and $T$ has no mass on $\{\varphi_{j}=-\infty\}$ for $j=1,2$. Let 
$$\varphi_{j,k}:= \max\{\varphi_{j}, -k \}$$
for every $j=1,2$ and $k \in \N$. Put $v_k:= \varphi_{1,k}- \varphi_{2,k}$. Put
$$Q_k:= d v_k \wedge \dc v_k \wedge T= \frac{1}{2}\ddc v_k^2  \wedge T - v_k\ddc v_k \wedge T.$$
By the plurifine locality with respect to $T$ (\cite[Theorem 2.9]{Viet-generalized-nonpluri}) applied to the right-hand side of the last equality, we have 
\begin{align}\label{eq-localplurifineddc}
\bold{1}_{\cap_{j=1}^2 \{\varphi_{j}> -k\}} Q_k =\bold{1}_{\cap_{j=1}^2\{\varphi_{j}> -k\}} Q_{s}
\end{align}
for every $s\ge k$. We say that $\langle d v \wedge \dc v \dot{\wedge} T \rangle$ is  \emph{well-defined} if the mass of $\bold{1}_{\cap_{j=1}^2 \{\varphi_{j}> -k\}} Q_k$ is uniformly bounded on $k$. In this case, using  (\ref{eq-localplurifineddc}) implies that there exists a positive current $Q$ on $X$ such that for every bounded Borel form $\Phi$ with compact support on $X$ such that 
$$\langle Q,\Phi \rangle  = \lim_{k\to \infty} \langle \bold{1}_{\cap_{j=1}^2 \{\varphi_{j}> -k\}} Q_k, \Phi\rangle,$$
and we define $\langle d v \wedge \dc v \dot{\wedge} T \rangle$ to be the current $Q$.  This agrees with the classical definition if $v$ is the difference of two  bounded quasi-psh functions. One can check  that this definition is independent of the choice of $\varphi_1, \varphi_2$.   By \cite[Lemma 2.5]{Viet-convexity-weightedclass}, if $v$ is bounded, then  $\langle d v \wedge \dc v \dot{\wedge} T \rangle$ is well-defined.

Let $w$ be another $T$-admissible dsh function.  If $T$ is of bi-degree $(n-1,n-1)$, we can also define the current $\langle  dv \wedge \dc w \dot{\wedge} T\rangle$ by a similar procedure as above. More precisely, we say $\langle dv \wedge \dc w \dot{\wedge} T \rangle$ is \emph{well-defined} if  $\langle dv \wedge \dc v \dot{\wedge} T \rangle$, $\langle dw \wedge \dc w \dot{\wedge} T \rangle$, and $\langle d(v+w) \wedge \dc (v+w) \dot{\wedge} T \rangle$ are well-defined. In this case, as in the classical case of bounded potentials, the defining formula for $\langle dv \wedge \dc w \dot{\wedge} T \rangle$ is obvious: 
$$2 \langle dv \wedge \dc w \dot{\wedge} T \rangle= \langle d(v+w) \wedge \dc (v+w) \dot{\wedge} T \rangle- \langle d v \wedge \dc v \dot{\wedge} T \rangle- \langle d w \wedge \dc w \dot{\wedge} T \rangle.$$
As above, if $v,w$ are bounded $T$-admissible, then $\langle dv \wedge \dc w \dot{\wedge} T \rangle$ is well-defined and given by the above formula. The following Cauchy-Schwarz inequality is clear from definition. 

\begin{lemma}\label{le-CSine} Assume that $\langle dv \wedge \dc w \dot{\wedge} T \rangle$ is well-defined. Then for every positive Borel function $\chi$,  we have
$$\int_X \chi \langle dv \wedge \dc w \dot{\wedge} T \rangle  \le \bigg(\int_X \chi \langle dv \wedge \dc v \dot{\wedge} T \rangle \bigg)^{1/2}\bigg(\int_X \chi \langle dw \wedge \dc w \dot{\wedge} T \rangle \bigg)^{1/2}.$$
\end{lemma}
We put
$$\langle \ddc v \dot{\wedge} T \rangle:= \langle \ddc \varphi_1 \dot{\wedge} T\rangle- \langle \ddc \varphi_2 \dot{\wedge} T\rangle$$
which is independent of the choice of $\varphi_1,\varphi_2$.  The following integration by parts formula is crucial for us later.

\begin{theorem} \label{th-integrabypart}  (\cite[Theorem 2.6]{Viet-convexity-weightedclass} or \cite[Theorem 3.1]{Vu_Do-MA}) Let  $T$ be a closed positive current of bi-degree $(n-1,n-1)$ on $X$. Let $v,w$ be bounded $T$-admissible dsh functions on $X$. If $\chi: \R \to \R$ is a $\cali{C}^3$ function then we have  
\begin{multline}\label{eq-intebypartschi}
\int_X \chi(w) \langle  \ddc v \dot{\wedge} T\rangle=\int_X v \chi''(w) \langle dw \wedge \dc w \dot{\wedge} T\rangle+\int_X v \chi'(w) \langle \ddc w \dot{\wedge} T\rangle\\
= - \int_X \chi'(w) \langle d w \wedge  \dc v \dot{\wedge} T\rangle.
\end{multline}
\end{theorem}

Since the case where $T$ is a non-pluripolar product of $(1,1)$-currents plays an important role in the study of the complex Monge-Amp\`ere equation, we present below an equivalent natural way to define the current $\langle d \varphi \wedge \dc \varphi \dot{\wedge} T \rangle$ in this setting. It is just for the purpose of clarification. 

\begin{lemma}\label{le-bangnhauTnonpluripolar} Let $u_1,\ldots, u_m$ be  negative psh functions on an open subset $U$ in $\C^n$ such that $T:=\langle \ddc u_1 \wedge \cdots \wedge \ddc u_m \rangle$ is well-defined. Let $v$ be the difference of two bounded psh functions on $U$.  For $k \in \N$, put $u_{j,k}:= \max\{u_j, -k\}$ and 
$$T_k:= \ddc u_{1,k} \wedge \cdots \wedge \ddc u_{m,k}.$$
 Then we have 
$$d v \wedge \dc v \wedge T= d v \wedge \dc v \wedge T_k$$
on $\cap_{j=1}^m \{u_j >-k\}$. 
\end{lemma}

\proof Put 
$$\psi_k:= k^{-1}\max \{u_1+ \cdots+u_m, -k\}+1.$$
Observe $\psi_k T_k= \psi_k T$. Now regularizing $v$ and using the continuity of Monge-Amp\`ere operators of bounded potentials, we obtain
$$\psi_k d v \wedge \dc v \wedge T= \psi_k d v \wedge \dc v \wedge T_k.$$
Hence 
\begin{align*}
d v \wedge \dc v \wedge T =  d v \wedge \dc v \wedge T_k
\end{align*}
on $U:=\cap_{j=1}^m \{u_j >-k/(2m)\}$ (for $\psi_k \ge 1/2$ on $U$). Note that $ d v \wedge \dc v \wedge T_k=  d v \wedge \dc v \wedge T_{k/(2m)}$ on $U$ by the plurifine locality. Thus the desired assertion follows.    This finishes the proof. 
\endproof

Let $T_1,\ldots, T_m$ be closed positive $(1,1)$-currents on $X$.  Let $n:= \dim X$.  \emph{Consider now} 
$$T:= \langle T_1 \wedge \cdots \wedge T_m \rangle.$$
Note that $T$ has no mass on pluripolar sets. Let $\varphi_1,\varphi_2$ be negative quasi-psh function on $X$.  Let $\varphi_{j,k}$ ($j=1,2$) be as before and $v:=\varphi_1- \varphi_2$.  In the moment, we work locally. Let $U$ be an open  small enough local chart (biholomorphic to a polydisk in $\C^n$) in $X$ such that  $T_j= \ddc u_j$ for $j=1,\ldots, m$, where $u_j$ is negative psh functions on $U$. Put $u_{j,k}:= \max\{u_j, -k\}$ for $k \in \N$, and 
$$T_k:= \ddc u_{1,k} \wedge \cdots \wedge \ddc u_{m,k}, \quad Q'_k:= d v_k \wedge \dc v_k \wedge T_k.$$
Put $A_k:= \cap_{j=1}^2 \{\varphi_{j}> -k\} \cap \cap_{j=1}^m \{u_j >-k\}$. By plurifine properties of Monge-Amp\`ere operators, we have 
$$\bold{1}_{A_k}Q'_k= \bold{1}_{A_{k}} Q'_{s}$$
for every $s \ge k$. One can check that the condition that  $(\bold{1}_{A_k} Q'_k)_k$ is of mass bounded uniformly (on compact subsets in $U$) in $k$ is independent of the choice of potentials. 

\begin{proposition}\label{pro-haidinhnghiaddctuongduong}  The current $\bold{1}_{A_k} Q'_k$ is of mass bounded uniformly in $k$ on compact subsets in $U$ for every $U$ (small enough biholomorphic to a polydisk in $\C^n$) if and only if the current $\langle d v \wedge \dc v \dot{\wedge} T \rangle$ is well-defined. In this case we have 
\begin{align}\label{eq-haidinhnghiatuongduong}
\langle d v \wedge \dc v \dot{\wedge} T \rangle= \lim_{k \to \infty} \bold{1}_{A_k} Q'_k.
\end{align}
\end{proposition}

\proof   By writing a smooth form of bi-degree $(n-m-1,n-m-1)$ as the difference of two smooth positive forms, we can assume without loss of generality that $T$ is of bi-degree $(n-1,n-1)$ (hence $m=n-1$).  Assume that $\langle d v \wedge \dc v \dot{\wedge} T \rangle$ is well-defined. We will check that $\bold{1}_{A_k} Q'_k$ is of mass bounded uniformly in $k$ on compact subsets in $U$. Let $\chi$ be a nonnegative smooth function compactly supported on $U$. Put
 $$\psi:= \varphi_1+\varphi_2+ u_1 + \cdots + u_m, \quad \psi_k:= k^{-1}\max\{\psi, -k\}+1.$$
 and $\varphi_{jk}:= \max\{\varphi_j, -k\}$ for $1 \le j \le 2$. Observe that $0 \le \psi_k \le 1$ and  if $\psi_k >0$, then $\varphi_j>-k$ for $1 \le j \le 2$; and 
\begin{align}\label{limit-psiknon}
\psi_k(x) \ge 1- s/k
\end{align}
for  every $x \in A_{s/(m+2)}$ and $1 \le s \le k$.   Recall $v_k:= \varphi_{1k}- \varphi_{2k}$ which is bounded (but not necessarily uniformly in $k$). 
 Observe that $\langle d v \wedge \dc v \dot{\wedge} T \rangle$ has no mass on pluripolar sets because $T$ is so (see for example \cite[Lemma 2.1]{Viet-generalized-nonpluri}). Put $Q''_k:= \psi_k Q_k= \psi_k \bold{1}_{A_k} Q'_k$. By (\ref{limit-psiknon}) and Lemma \ref{le-bangnhauTnonpluripolar}, we have
\begin{align}\label{ine-dvdcvTlimitk} 
\langle d v \wedge \dc v \dot{\wedge} T \rangle &= \lim_{k\to \infty} \psi_k d v_k \wedge \dc v_k \wedge T\\
\nonumber
&= \lim_{k\to \infty} \psi_k d v_k \wedge \dc v_k \wedge T_k =\lim_{k\to \infty}Q''_k
\end{align}
on $U$. On the other hand, by (\ref{limit-psiknon}) again,  we see that the claim that $Q''_k$ is of mass uniformly bounded on compact subsets in $U$ is equivalent to that $\bold{1}_{A_k}Q'_k$ is so. This together with (\ref{ine-dvdcvTlimitk}) yields the desired assertion. 

Conversely, suppose now that  $\bold{1}_{A_k} Q'_k$ is of mass bounded uniformly in $k$ on compact subsets in $U$ for every $U$. Thus there exists a positive current $R$ on $U$ such that $\bold{1}_{A_k} R= \bold{1}_{A_k} Q'_k$ for every $k$ and $U$. Set 
$$\tilde{\psi} := \varphi_1+ \varphi_2, \quad \tilde{\psi}_k:= k^{-1} \max\{\tilde{\psi}, -k\}+1.$$
Let $s \in \N$ with $s \ge k$. Observe
\begin{align*}
\psi_{s} R=\tilde{\psi}_k  \psi_{s} R+ (1- \tilde{\psi}_k) \psi_{s} R.  
\end{align*}
The second term in the right-hand side of the last inequality tends to $0$ (uniformly in $s$) because $\tilde{\psi}_k$ converges pointwise to $1$ outside a pluripolar set and $R$ has no mass on pluripolar sets. Using Lemma \ref{le-bangnhauTnonpluripolar}, we have 
\begin{align*}
\tilde{\psi}_k  \psi_{s} R &= \tilde{\psi}_k  \psi_{s}  d v_{s} \wedge \dc v_{s} \wedge T_{s} \\
&= \tilde{\psi}_k  \psi_{s}  d v_{s} \wedge \dc v_{s} \wedge T= \tilde{\psi}_k  \psi_{s}  d v_{k} \wedge \dc v_{k} \wedge T,
\end{align*} 
here we used the plurifine topology properties with respect to $T$ (see \cite[Theorem 2.9]{Viet-generalized-nonpluri}), thanks to the fact that   $\varphi_{j,k} = \varphi_{j, s}$ on $\{\tilde{\psi}_k \not = 0\}$ for $j=1,2$ (recall $s \ge  k$), and they are bounded psh functions. Letting $s \to \infty$ gives
$$\tilde{\psi}_k R = \tilde{\psi}_k  \bold{1}_{\cup_{j=1}^m \{u_j > - \infty\}}  d v_{k} \wedge \dc v_{k} \wedge T=\tilde{\psi}_k  d v_{k} \wedge \dc v_{k} \wedge T$$
because the current $d v_{k} \wedge \dc v_{k} \wedge T$ has no mass on pluripolar sets. Now letting $k \to \infty$ gives the desired assertion. This finishes the proof.
\endproof

Thanks to Proposition \ref{pro-haidinhnghiaddctuongduong}, we can use the right-hand side of (\ref{eq-haidinhnghiatuongduong}) to define $\langle d v \wedge \dc v \dot{\wedge} T \rangle$ in the case where $T$ is the non-pluripolar product of some closed positive $(1,1)$-currents. By the same reason, in this case, we will use the expression $dv \wedge \dc w \wedge T_1 \wedge \ldots \wedge T_{n-1}$ to denote $\big\langle d v \wedge \dc w \dot{\wedge} \langle T_1 \wedge \cdots \wedge T_{n-1}\rangle \big\rangle$ whenever it is well-defined.

\subsection{Auxiliary facts on weights} \label{subsec-auxi}

In this subsection, we present some facts about weights needed for the proofs of main results.

Recall that  $\widetilde{\mathcal{W}}^-$ is the set of all convex, non-decreasing functions 
$\chi: \R_{\le 0}\rightarrow\R_{\le 0}$ such that $\chi(0)=0$ and $\chi \not \equiv 0$.  Let $M \ge 1$ be a constant and  $\mathcal{W}^+_M$ the usual space of  increasing concave functions $\chi: \R_{\le 0} \to \R_{ \le 0}$ such that $\chi(0)=0$, $\chi<0$ on $(-\infty, 0)$, and $|t \chi'(t)| \le M|\chi(t)|$ for every $t \le 0$. 
 We have the following lemmas. 

\begin{lemma} \label{lem approchi} Let $c>0$, $0<\delta<1$ and $\chi: \R\rightarrow\R$ such that 
	$\chi(t)=ct$ for every $t\geq-\delta$ and $\chi|_{(-\infty, 0]}\in \widetilde{\mathcal{W}}^-\cup \mathcal{W}^+_M$ ($M\geq 1$).  Let $g$ be a smooth radial cut-off function supported in $[-1,1]$ on $\R$, \emph{i.e,} $g(t)= g(-t)$ for $t \in \R$, $0 \le g \le 1$ and $\int_\R g(t) dt =1$. Put $g_\epsilon(t):=  \epsilon^{-1}g(\epsilon t)$ for every constant $\epsilon >0$ and $\chi_\epsilon:= \chi * g_\epsilon$ (the convolution of $\chi$ with $g_\epsilon$). Then  the following assertions are true:
	
	$(i)$ if $\chi|_{(-\infty, 0]} \in \widetilde{\mathcal{W}}^-$, then   $\chi_\epsilon|_{(-\infty, 0]}\in \widetilde{\mathcal{W}}^-$
	for every $0<\epsilon<\delta$,  $\chi_\epsilon\searrow \chi$ as $\epsilon\searrow 0$
	and $\sup (\chi_{\epsilon}-\chi)\leq c\epsilon$;
	
	$(ii)$ if $\chi|_{(-\infty, 0]} \in \mathcal{W}^+_M$ and $0<\epsilon<\delta^2/2$ then  
	$\chi_\epsilon|_{(-\infty, 0]}\in \mathcal{W}^+_{M/(1-\delta)}$. Moreover, if $0<\epsilon<\delta^2/8$ then
	$$\overline{\chi}_{\epsilon}:=\chi_{\epsilon}( \cdot +\epsilon)-c\epsilon\in \mathcal{W}^+_{M/(1-\delta)^2}, \quad 
	\overline{\chi}_{\epsilon}\geq\chi-c\epsilon,$$
	 and $\overline{\chi}_{\epsilon}$ converges uniformly to $\chi$ as $\epsilon \to 0$ on compact subsets in $\R$.
\end{lemma}
\begin{proof}
	The part $(i)$ follows from \cite[Lemma 2.1]{Vu_Do-MA}. The part $(ii)$ can be obtained more or less by similar arguments as in the last reference. We provide details for readers' convenience. 	
	It is clear that $\chi_{\epsilon}$ is a concave, increasing function with $\chi_{\epsilon}(0)=0$.
	 We will show that
	 \begin{equation}\label{eq1 lem2.7}
	 	\chi_{\epsilon}'(t)\leq\dfrac{M}{1-\delta}\dfrac{\chi_{\epsilon}(t)}{t},
	 \end{equation}
	for every $t<0$ and $0<\epsilon<\delta^2/2$.
	
	If $t<-\dfrac{\delta}{2}$ then we have
	\begin{align*}
	\chi_{\epsilon}'(t)=\int_{-\epsilon}^{\epsilon}\chi'(t-s)g_{\epsilon}(s)ds
		\leq\int_{-\epsilon}^{\epsilon}\dfrac{M \chi (t-s)}{t-s}g_{\epsilon}(s)ds
		&\leq \int_{-\epsilon}^{\epsilon}\dfrac{M \chi (t-s)}{t+\epsilon}g_{\epsilon}(s)ds\\
		&=\dfrac{M\chi_{\epsilon}(t)}{t+\epsilon}\\
		&=\dfrac{M t}{t+\epsilon}\dfrac{\chi_{\epsilon}(t)}{t}\\
		&\leq \dfrac{M}{1-\delta}\dfrac{\chi_{\epsilon}(t)}{t},
	\end{align*}
	for every $0<\epsilon<\delta^2/2$.
	
On the other hand, if $t\geq -\dfrac{\delta}{2}$, then  $\chi_{\epsilon}(t)=\chi(t)=ct$ for every  $0<\epsilon<\delta^2/2$. As a consequence, we have
$$\chi_{\epsilon}'(t)=\chi'(t)\leq \dfrac{M\chi(t)}{t}=M\dfrac{\chi_{\epsilon}(t)}{t}.$$
Thus, \eqref{eq1 lem2.7} follows.	Hence, $\chi_\epsilon|_{(-\infty, 0]}\in \mathcal{W}^+_{M/(1-\delta)}$.

Now, we consider $\overline{\chi}_{\epsilon}$. Since $\chi$ is increasing, one sees that 	$\overline{\chi}_{\epsilon}\geq\chi-c\epsilon$ and $\overline{\chi}_{\epsilon}$ converges uniformly to $\chi$ as $\epsilon \to 0$ on compact subsets in $\R$. It remains
to show that
$\overline{\chi}_{\epsilon}\in \mathcal{W}^+_{M(1+\delta)/(1-\delta)}$ for every $0<\epsilon<\delta^2/8$.
Note that
$$\overline{\chi}_{\epsilon}=h_{\epsilon}* g_{\epsilon},$$
where $h_{\epsilon}(t)=\chi(t+\epsilon)-c\epsilon$. The function $\overline{\chi}_\epsilon(t)$ is concave, increasing
and $\overline{\chi}+\epsilon(0)=0$.

 If $-\delta/2\leq t<0$ then  $h_{\epsilon}(t)=\chi(t)=ct$ for every $0<\epsilon<\delta^2/2$. Therefore
 $$h_{\epsilon}'(t)=\chi'(t)\leq \dfrac{M\chi(t)}{t}=M\dfrac{h_{\epsilon}(t)}{t}.$$
 If $t<-\delta/2$ then 
 \begin{align*}
 	h_{\epsilon}'(t)=\chi'(t+\epsilon)\leq M\dfrac{\chi (t+\epsilon)}{t+\epsilon}
 	\leq M\dfrac{\chi(t+\epsilon)-c\epsilon}{t+\epsilon}
 	=M\dfrac{h_{\epsilon}(t)}{t+\epsilon}
 	&=\dfrac{M t}{t+\epsilon}\dfrac{h_{\epsilon}(t)}{t}\\ 	
 	&\leq \dfrac{M }{1-\delta}\dfrac{h_{\epsilon}(t)}{t},
 \end{align*}
 for every $0<\epsilon<\delta^2/2$.
 
 Then, for every $0<\epsilon<\delta^2/2$, we have $h_{\epsilon}\in \mathcal{W}^+_{M/(1-\delta)}$ 
 and $h_{\epsilon}=ct$ for every $t\geq-\delta/2$. Hence, for every $0<\epsilon<\delta^2/8$, we have
 $$\overline{\chi}_{\epsilon}=h_{\epsilon}* g_{\epsilon}\in 
  \mathcal{W}^+_{\frac{M}{(1-\delta)(1-\delta/2)}}
  \subset  \mathcal{W}^+_{\frac{M}{(1-\delta)^2}}.$$
 The proof is completed.
\end{proof}

\begin{lemma}\label{le-regularizedchi}
	Let $\chi, \tilde{\chi}\in \widetilde{\mathcal{W}}^-\cup \mathcal{W}^+_M$ ($M \ge 1$) such that $\tilde{\chi} \le \chi$.	Then, there exist sequences of functions
	$\chi_j, \tilde{\chi}_j\in\widetilde{\mathcal{W}}^-\cup \mathcal{W}^+_{M_j}$
(with $M_j\searrow M$ as $j\to\infty$ )	satisfying the following conditions:
	\begin{itemize}
		\item 	$\chi_j\in \Cc^{\infty}(\R)$ for every $j$;
		\item  $\chi_j\geq\tilde{\chi}_j$ and $\chi_j\geq\chi-2^{-j}$ for every $j$ big enough;
		\item  $\tilde{\chi}-2^{-j}\leq\tilde{\chi}_j\leq\tilde{\chi}$ on $(-\infty, -1]$ for every $j$ big enough;
		\item  	$\chi_j$ converges uniformly to $\chi$ on compact subsets in $\R_{\le 0}$.
	\end{itemize}
\end{lemma}
\begin{proof}
 We  split the proof into two cases.\\
 
 \noindent
{\bf Case 1}:  $\chi\in \widetilde{\mathcal{W}}^-$.\\
For every $j\geq 1$, we denote 
$$\overline{\chi}_j(t)=\begin{cases}
	\max\{\chi(t), c_jt\}\qquad\mbox{if}\quad t<0,\\
	c_jt\qquad\mbox{if}\quad t\geq 0,
\end{cases}$$
where $$c_j:=\dfrac{-\chi(-2^{-j})}{2^{-j}} \cdot$$
Then $\overline{\chi}_j$ satisfies the hypothesis of Lemma \ref{lem approchi} for $\delta:= 2^{-j}$.
Let $g$ be a smooth radial cut-off function supported in $[-1,1]$ on $\R$, \emph{i.e,} $g(t)= g(-t)$ for $t \in \R$, $0 \le g \le 1$ and $\int_\R g(t) dt =1$. For every $j\geq 1$, we define
\begin{center}
	$\chi_j=\overline{\chi}_j * g_{4^{-j-1}}\quad$ and $\quad \tilde{\chi}_j=\tilde{\chi}$.
\end{center}
By Lemma \ref{lem approchi}, we have $\chi_j$ and $\tilde{\chi}_j$ satisfy the desired conditions.\\

\noindent
{\bf Case 2}:  $\chi\in\mathcal{W}^+_M$.\\
Since $\chi \ge \tilde{\chi}$, we also have $\tilde{\chi}\in\mathcal{W}^+_M$.
Assume that $g$ and $c_j$ are as in Case 1.
 For every $j\geq 1$, we define
 $$\overline{\chi}_j(t)=\begin{cases}
 	\min\{\chi(t), c_jt\}\qquad\mbox{if}\quad t<0,\\
 	c_jt\qquad\mbox{if}\quad t\geq 0,
 \end{cases}$$
and
	$$\chi_j(t)=(\overline{\chi}_j(\cdot+4^{-j-1}) * g_{4^{-j-1}})(t)-c_j4^{-j-1}.$$
We also denote
 $\tilde{\chi}_j(t)=\min\{\tilde{\chi}(t), \chi_j(t)\}$.
  By  Lemma \ref{lem approchi}, we have $\chi_j$ and $\tilde{\chi}_j$ satisfy the desired conditions. The proof is completed.
\end{proof}

Let $\phi$ be a negative $\theta$-psh function. We denote by $\PSH(X, \theta, \phi)$ the set of $\theta$-psh functions $u \le \phi$. Recall that by monotonicity, we always have $\int_X \theta^n_u \le \int_X \theta^n_\phi$, where for every $\theta$-psh function $v$, we put $\theta_v:= \ddc v +\theta$.  We also define by $\mathcal{E}(X,\theta, \phi)$ the set of $u \in \PSH(X, \theta, \phi)$ of full Monge-Amp\`ere mass with respect to $\phi$, \emph{i.e,} $\int_X\theta_u^n = \int_X \theta_\phi^n.$ 

Let $\chi \in \widetilde{\mathcal{W}}^- \cup \mathcal{W}^+_M$, and $u \in \PSH(X, \theta, \phi)$. We put
$$E_{\chi, \theta,\phi}(u):= \int_X - \chi(u- \phi) \theta_u^n.$$
We also define by $\mathcal{E}_\chi(X, \theta,\phi)$ the set of $u \in \mathcal{E} (X, \theta, \phi)$ with $E_{\chi, \theta, \phi}(u)<\infty$.

\begin{lemma} \label{le-sosanhnangnluongintegrabig} Let $\chi \in \widetilde{\mathcal{W}}^-\cup \mathcal{W}^+_M$ and $u_1,u_2 \in \mathcal{E}_\chi(X, \theta, \phi)$. Then there exists a constant $C_1>0$ depending only on $n$ and $M$
	 such that
$$- \int_X \chi(u_1- \phi) \theta_{u_2}^n \le C_1 \sum_{j=1}^2 E_{\chi, \phi, \theta}(u_j),$$
and 
$$E_{\chi,\theta,\phi}\big(au_1+(1-a)u_2\big) \le C_1 \sum_{j=1}^2 E_{\chi,\theta,\phi}(u_j),$$
for every $0<a<1$.
Furthermore if $u_1 \ge u_2$, then 
$$E_{\chi, \phi, \theta}(u_1) \le C_2 E_{\chi, \phi,\theta} (u_2),$$
for some constant $C_2$ depending only on $n$ and $M$.
\end{lemma}

\proof The first and third inequalities are  from \cite[Lemma 3.2]{Vu_Do-MA} (see also \cite[Propositions 2.3, 2.5]{GZ-weighted} for the case where $\phi=0$ and $\theta$ is a K\"ahler form). The second desired inequality was  implicitly proved in  the proof of convexity of finite energy classes in \cite[Proposition 3.3]{Viet-convexity-weightedclass} (in a much broader context). Alternatively one can use properties of envelopes in \cite{Lu-Darvas-DiNezza-mono} to get the same conclusion.  We prove here the second desired inequality using ideas from \cite{Viet-convexity-weightedclass} for readers' convenience. We only consider $\chi \in \widetilde{\mathcal{W}}^-$. The case where $\chi \in \mathcal{W}^+_M$ is done similarly.

 Considering $u_j - \epsilon$ for $\epsilon>0$ instead of $u_j$, and taking $\epsilon \to 0$ later, without loss of generality, we can assume that $u_j <\phi \le 0$ for $j=1.2$. By replacing $u_j,\theta$ by $u_j- \phi$, $\theta_\phi$ respectively, we can assume that $\phi =0$, but $\theta$ is no longer a smooth form but a closed positive $(1,1)$-current. This change causes no trouble for us.  Let $v:= au_1+(1-a)u_2$.  Observe that $X \subset \{u_1 < u_2\} \cup \{u_1 > 2 u_2\}$. 
Hence 
\begin{align*}
	E_{\chi,\theta}(v)
 &\le \int_{\{u_1 < u_2\}} -\chi(v) \theta_v^n +\int_{\{u_1 > 2u_2\}} -\chi(v) \theta_v^n\\
& \le  \sum_{k=0}^{n}\left(\int_{\{u_1 < u_2\}} -\chi(u_1) \theta_{u_1}^k \wedge \theta_{u_2}^{n-k} +\int_{\{u_1 > 2u_2\}}-\chi((1+a)u_2) \theta_{u_1}^k \wedge \theta_{u_2}^{n-k}\right)\\
& \le  \sum_{k=0}^{n}\int_{\{u_1 < u_2\}} -\chi(u_1) \theta_{u_1}^k \wedge \theta_{\max\{u_1,u_2\}}^{n-k} +\\
&\quad + \sum_{k=0}^n \int_{\{u_1 > 2u_2\}}-2^{k+1} \chi(u_2) \theta_{\max\{u_1/2, u_2\}}^k \wedge \theta_{u_2}^{n-k}\\
& \le  \sum_{k=0}^{n}\left(\int_X -\chi(u_1) \theta_{u_1}^k \wedge \theta_{\max\{u_1,u_2\}}^{n-k} +2^{k+1} \int_X- \chi(u_2) \theta_{\max\{u_1/2, u_2\}}^k \wedge \theta_{u_2}^{n-k}\right)\\
& \lesssim  E_{\chi,\theta}(u_1)+E_{\chi,\theta}(\max\{u_1, (u_1+u_2)/2\})+ E_{\chi,\theta}(u_2)+  E_{\chi, \theta}(\max\{u_1/4+u_2/2, u_2\}) \\
& \lesssim  E_{\chi,\theta}(u_1)+ E_{\chi,\theta}(u_2),
\end{align*}
where the two last estimates hold due to  the first and third inequalities of the lemma.
 This finishes the proof. 
\endproof

\begin{lemma}\label{le-uocluongchiepsilon} 
Let $\chi, \tilde{\chi} \in \widetilde{\mathcal{W}}^-\cup \mathcal{W}^+_M$ such that $\tilde{\chi} \le \chi$ and 
let $u_1, u_2,..., u_{n+1} \in \mathcal{E}(X, \theta,\phi)$. Denote $\varrho:= \vol(\theta_\phi)$. Then there exists a constant $C>0$ depending only on $n$ and $M$ such that  
$$- \int_X \chi(\epsilon (u_1-\phi))\theta_{u_2}\wedge...\wedge\theta_{u_{n+1}} \le C\, B\varrho  (1-\tilde{\chi}(-1))Q_{0}(\epsilon),$$
for every $0<\epsilon\leq 1$, where 
\begin{center}
	$B=1+\max_{1\leq j\leq n+1}E_{\tilde{\chi}, \theta, \phi}(u_j)/\varrho\quad$
	and $\quad Q_{0}(\epsilon):= \sup_{\{t \le -1\}}\dfrac{\chi(\epsilon t)}{\tilde{\chi}(t)}\cdot$
\end{center}

\end{lemma}

\proof  Let $L$ be the left-hand side of the desired inequality. 
We have
\begin{align*}
L  & \le - \int_{\{u_1 \ge\phi-1\}} \chi(\epsilon (u_1-\phi))  \theta_{u_2}\wedge...\wedge\theta_{u_{n+1}} 
- \int_{\{u_1 < \phi-1\}} \chi(\epsilon (u_1-\phi))  \theta_{u_2}\wedge...\wedge\theta_{u_{n+1}}\\
& \le  -\chi(-\epsilon)\varrho- Q_{0}(\epsilon)\int_{\{u_1 <\phi -1\}} \tilde{\chi}(u_1-\phi) \theta_{u_2}\wedge...\wedge\theta_{u_{n+1}}\\
&\le -\varrho Q_{0}(\epsilon)\tilde{\chi}(-1) - Q_{0}(\epsilon)\int_{X} \tilde{\chi}(u_1-\phi) \theta_{u_2}\wedge...\wedge\theta_{u_{n+1}}\\
&\le  -\varrho Q_{0}(\epsilon)\tilde{\chi}(-1)+C Q_{0}(\epsilon)\max_{1\leq j\leq n+1}E_{\tilde{\chi}, \theta, \phi}(u_j),\\
\end{align*}
where $C>0$ depends only on $n$ and $M$. The last estimate holds due to Lemma \ref{le-sosanhnangnluongintegrabig}.
Thus the desired inequality follows. 
\endproof

By the convexity/concavity and by the assumption $\tilde{\chi}\leq\chi$, we have
\begin{equation}
	\begin{cases}
		Q_{0}(\epsilon)\geq\epsilon Q_{0}(1)\quad\mbox{if}\quad\chi\in\widetilde{\mathcal{W}}^{-},\\
		Q_{0}(\epsilon)\leq\epsilon Q_{0}(1)\quad\mbox{if}\quad\chi\in\mathcal{W}_M^{+},
	\end{cases}
\end{equation}
for every $0<\epsilon \le 1$.
  Moreover, if $\chi\in\widetilde{\mathcal{W}}^{-}$ and $\chi(t)/\tilde{\chi}(t)\rightarrow 0$ as $t\rightarrow-\infty$, then by the definition of $Q_{0}$, we also have
\begin{equation}\label{eq chantrenQ2}
	Q_{0}(\epsilon)\leq \dfrac{\chi (-\epsilon^{1/2})}{\tilde{\chi}(-1)}
	+\sup_{\{t \le -\epsilon^{-1/2}\}}\dfrac{\chi(t)}{\tilde{\chi}(t)}\stackrel{\epsilon\to 0^+}{\longrightarrow} 0.
\end{equation}

Let $u_1,u_2 \in \mathcal{E}_\chi(X, \theta, \phi)$, and   
  $v:= \max\{u_1, u_2\}$. Put
$$\nu(u_1,u_2):= \chi(-|u_1- u_2|) (\theta_{u_2}^n- \theta_{u_1}^n),$$
and 
\begin{align} \label{eq-dinhgnhiaIchi}
I_\chi(u_1,u_2) &:= \int_{\{u_1< u_2\}} \nu(u_1,u_2)+\int_{\{u_1> u_2\}} \nu(u_2,u_1) = \int_X \nu(u_1,v) + \int_X \nu(u_2,v).
\end{align}

\begin{proposition}\label{prop Ichi positive}
	Let $\chi \in \widetilde{\mathcal{W}}^- \cup \mathcal{W}^+_M$.
Let $\phi$ is a negative $\theta$-psh function and $u_1, u_2\in \mathcal{E}_\chi(X, \theta,\phi)$.
	Then
	$$I_\chi(u_1,u_2)\geq 0.$$
\end{proposition}

\begin{proof}
	Denote $\mu=\theta_{u_2}^n- \theta_{u_1}^n$.
Since $\chi$ is absolutely continuous, we have $\chi$ is differentiable almost everywhere and 
$-\chi(t)=\int_t^0\chi'(s)ds$ for every $t<0$.
Hence
	\begin{align*}
		\int_{\{u_1< u_2\}} \nu(u_1,u_2)&=-\int_{\{u_1< u_2\}}\left(\int_{u_1-u_2}^0\chi'(t)dt\right)d\mu\\
		&=-\int_{\{u_1< u_2\}}\left(\int_{-\infty}^0\chi'(t)\mathbf{1}_{\{u_1<u_2+t\}}dt\right)d\mu\\
		&=-\int_{-\infty}^0\chi'(t)\mu\{u_1<u_2+t\}dt.
\end{align*}
Moreover, it follows from \cite[Lemma 2.3]{Lu-Darvas-DiNezza-logconcave} that $\mu\{u_1<u_2+t\}\leq 0$
for every $t\leq 0$. Hence
$$\int_{\{u_1< u_2\}} \nu(u_1,u_2)=-\int_{-\infty}^0\chi'(t)\mu\{u_1<u_2+t\}dt\geq 0.$$
Similarly, we have
$$\int_{\{u_2< u_1\}} \nu(u_2, u_1)\geq 0.$$
Thus $$I_\chi(u_1,u_2)= \int_{\{u_1< u_2\}} \nu(u_1,u_2)+\int_{\{u_2< u_1\}} \nu(u_2, u_1) \geq 0.$$
\end{proof}

\section{Stability for weighted potentials} \label{sec-fixedtype}

\subsection{Main results} 

Let $\chi, \tilde{\chi} \in \widetilde{\mathcal{W}}^-\cup\mathcal{W}_M^+$ ($M\geq 1$) such that $\tilde{\chi} \le \chi$.   For each constant $t\geq 0$,
 we denote
\begin{equation}\label{eq QB}
		Q(t)= Q_{\chi,\tilde{\chi}}(t):=
	\begin{cases}
		1 \quad\mbox{if}\quad t\ge 1,\\
			\left(Q_{0}(t)/Q_{0}(1)\right)^{1/2}\quad\mbox{if}\quad 0<t<1\quad\mbox{and}\quad\chi\in\widetilde{\mathcal{W}}^{-},\\
		t^{1/2}\quad\mbox{if}\quad 0<t<1\quad\mbox{and}\quad\chi\in\mathcal{W}_M^{+},\\
		\lim_{s\to 0^+}Q(s)\quad\mbox{if}\quad t=0.
	\end{cases}
\end{equation}
where $Q_{0}$ is defined as in Lemma \ref{le-uocluongchiepsilon}. We remove the subscript $\chi,\tilde{\chi}$ from $Q_{\chi,\tilde{\chi}}$ if $\chi, \tilde{\chi}$ are clear from the context.  Note that $Q$ is increasing continuous function in $t$ and 
\begin{equation}
	Q(0)=0\quad\mbox{if either }\quad \chi \in \mathcal{W}_M^+\quad\mbox{or}\quad
	\lim\limits_{t\to -\infty}\dfrac{\chi(t)}{\tilde{\chi}(t)}=0.
\end{equation}

For the convenience, we normalize energies with respect to $\varrho:= \int_X \theta_\phi^n$ as follows
  $$E^0_{\tilde{\chi}, \theta, \phi}:=\varrho^{-1}E_{\tilde{\chi}, \theta, \phi}, \quad I^0_\chi(u_1,u_2)=\varrho^{-1}I_\chi(u_1,u_2).$$

\begin{theorem} \label{th-lowerenergy}  Let $\theta$ be a closed smooth real $(1,1)$-form  and $\phi$ be a  negative $\theta$-psh function such that $\varrho:=\vol(\theta_\phi)>0$.  Let $\chi, \tilde{\chi} \in \widetilde{\mathcal{W}}^-\cup\mathcal{W}^+_M$ ($M\geq 1$) such that $\tilde{\chi} \le \chi$. Let $B \ge 1$ be a constant and let $u_j, \psi_j \in \mathcal{E}(X, \theta, \phi)$ satisfy $u_1 \le u_2$ and 
$$E^0_{\tilde{\chi}, \theta, \phi}(u_j)+E^0_{\tilde{\chi}, \theta, \phi}(\psi_j) \le B,$$
for $j=1,2$. Then there exists a constant $C_n>0$ depending only on $n$ and $M$ such that 
\begin{align}\label{ine-thlowerenergy}
\int_X  -\chi (u_1-u_2)  (\theta_{\psi_1}^n- \theta_{\psi_2}^n) \le 
C_n\varrho B^2(1-\tilde{\chi}(-1))^2 Q^{\circ n}\big(I^0_\chi(u_1,u_2)\big),
\end{align}
where   $Q$ is defined by \eqref{eq QB}, and $Q^{\circ n}:= Q \circ Q \circ \cdots \circ Q$ ($n$-iterate of $Q$).
\end{theorem}

Since the measure $\theta_{\psi_1}^n- \theta_{\psi_2}^n$ is not positive, we need the following consequence of the above theorem for later applications on stability estimates.

\begin{theorem} \label{th-lowerenergy-kocochuanhoa}
	 Let $\theta$ be a closed smooth real $(1,1)$-form  and $\phi$ be a  negative $\theta$-psh function such that  $\phi=P_{\theta}[\phi]$, $\varrho:=\vol(\theta_\phi)>0$
	 and $\theta\leq A\omega$ for some constant $A \ge 1$.  Let 
	 $\chi, \tilde{\chi} \in \widetilde{\mathcal{W}}^-\cup\mathcal{W}^+_M$ ($M\geq 1$)
	 such that $\tilde{\chi} \le \chi$. Let $B \ge 1$ be a constant and $u_1, u_2, \psi \in \mathcal{E}(X, \theta, \phi)$ satisfying
	$$
	E^0_{\tilde{\chi}, \theta, \phi}(u_1)+
	E^0_{\tilde{\chi}, \theta, \phi}(u_2)+E^0_{\tilde{\chi}, \theta, \phi}(\psi) \le B,$$
	for $j=1,2$. Then, for every constant $m>0$ and $0<\gamma<1$, there exists a constant $C>0$ depending on $n, M, X, \omega, m$ and $\gamma$  such that 
\begin{align*}
\int_X  -\chi\big(-|u_1- u_2|\big) \theta_\psi^n \le -\varrho \chi\left(-|a_1-a_2|-\lambda^m\right)
+C\varrho A^{(1-\gamma)/m}(B-\tilde{\chi}(-A))^2(1-\tilde{\chi}(-1))^2\lambda^{\gamma},
\end{align*}
where $a_j:= \sup_Xu_j$ and $\lambda=Q^{\circ n}\big(I^0_\chi(u_1,u_2)\big)$.
\end{theorem}



\subsection{Proof of Theorem \ref{th-lowerenergy}}  \label{subsec-key}

Here is the first step in the proof of Theorem \ref{th-lowerenergy}.

\begin{lemma} \label{le-lientuccungkididwedgedc} If  Theorem \ref{th-lowerenergy} holds for $u_j, \psi_j$ of the same singularity type as $\phi$, then it holds for the general case.
\end{lemma}

\proof Let $u_j, \psi_j$ ($j=1,2$) be as in the statement of Theorem \ref{th-lowerenergy}.  For every
$k>0$, we denote $u_{j,k}:= \max\{u_j,  \phi -k\}$
 and $\psi_{j, k}=\max\{\psi_j, \phi-k\}$. By Lemma \ref{le-sosanhnangnluongintegrabig}, there exists a constant $C_1>0$ depending
only on $n$ and $M$ such that
$$ E^0_{\tilde{\chi},\theta, \phi}(u_{j,k})+ E^0_{\tilde{\chi},\theta, \phi}(\psi_{j,k})\leq C_1B,$$
for $j=1, 2$ and for every $k>0$. Therefore, by the assumption, there exists a constant $C_2>0$
depending only on $n$ and $M$ such that
$$\int_X  -\chi (u_{1, k}-u_{2, k})  (\theta_{\psi_{1, l}}^n- \theta_{\psi_{2, l}}^n) \le 
C_2\varrho B^2(1-\tilde{\chi}(-1))^2 Q^{\circ (n)}(I^0_\chi(u_{1, k},u_{2, k})),$$
for every $k, l>0$. Letting $l\rightarrow\infty$ and using \cite[Theorem 2.2]{Darvas-Lu-DiNezza-singularity-metric}, we get
\begin{equation}\label{eq0 le-lientuccungkididwedgedc}
	\int_X  -\chi (u_{1, k}-u_{2, k})  (\theta_{\psi_{1}}^n- \theta_{\psi_{2}}^n) \le 
	C_2\varrho B^2(1-\tilde{\chi}(-1))^2 Q^{\circ (n)}(I^0_\chi(u_{1, k},u_{2, k}))
\end{equation}
for every $k>0$.
 We will show that
\begin{equation}\label{eq1 le-lientuccungkididwedgedc}
	I_\chi(u_1,u_2)=\lim_{k \to \infty}I_{\chi}(u_{1,k}, u_{2,k}).
\end{equation}
Denote
$$f:= \chi(u_{1}- u_{2}) (\theta_{u_{2}}^n - \theta_{u_{1}}^n), \quad f_k:=  \chi(u_{1,k}- u_{2,k}) (\theta_{u_{2,k}}^n - \theta_{u_{1,k}}^n).$$
We have 
\begin{align*}
	I_{\chi}(u_{1,k}, u_{2,k})  = \int_{X} f_k 
	&= \int_{\{u_1> \phi-k\}}f_k+ \int_{\{u_1 \le \phi-k\}}  f_k\\
	& =\int_{\{u_1> \phi-k\}} f+ \int_{\{u_1 \le \phi  -k\}}  f_k\\
	&= I_\chi(u_1,u_2)- \int_{\{u_1\le \phi -k\}}f + \int_{\{u_1 \le  \phi -k\}}  f_k.
\end{align*}
Then
\begin{align*}
	|I_{\chi}(u_{1,k}, u_{2,k})-I_\chi(u_1,u_2)|&= 
	\bigg|\int_{\{u_1\le \phi -k\}}f - \int_{\{u_1 \le  \phi -k\}}  f_k\bigg|\\
	&\leq\int_{\{u_1\le \phi -k\}}\mu+ \int_{\{u_1 \le  \phi -k\}} -\chi(u_{1,k}- u_{2,k}) (\theta_{u_{2,k}}^n + \theta_{u_{1,k}}^n)\\
	&\leq \int_{\{u_1\le \phi -k\}}\mu+ \int_{\{u_1 \le  \phi -k\}} -\chi(-k) (\theta_{u_{2,k}}^n + \theta_{u_{1,k}}^n),
\end{align*}
where $\mu=-\chi(u_1-\phi)(\theta_{u_1}^n+\theta_{u_2}^n)$.
By Lemma \ref{le-sosanhnangnluongintegrabig}, we have  $\int_X\mu<\infty$. Then it follows from Lebesgue's
dominated convergence theorem that $\lim_{k\to\infty}\int_{\{u_1\le \phi -k\}}\mu=0$. Therefore, 
\begin{equation}\label{eq2 le-lientuccungkididwedgedc}
	\limsup\limits_{k\to\infty}	|I_{\chi}(u_{1,k}, u_{2,k})-I_\chi(u_1,u_2)|
	\leq \limsup\limits_{k\to\infty}\int_{\{u_1 \le  \phi -k\}} -\chi(-k) (\theta_{u_{1,k}}^n + \theta_{u_{2,k}}^n).
\end{equation} 
By the fact that 
 $$\int_X\theta_{u_{1,k}}^n=\int_X\theta_{u_{2,k}}^n=\int_X\theta_{\phi}^n, \quad \mathbf{1}_{\{u_1 >  \phi -k\}}\theta_{u_{j,k}}^n=\mathbf{1}_{\{u_1 >  \phi -k\}}\theta_{u_{j}}^n \quad (j=1,2),$$
  we have
\begin{equation}\label{eq3 le-lientuccungkididwedgedc}
	-\chi(-k)\int_{\{u_1 \le  \phi -k\}} (\theta_{u_{1,k}}^n + \theta_{u_{2,k}}^n)
	= -\chi(-k)\int_{\{u_1 \le  \phi -k\}} (\theta_{u_{1}}^n + \theta_{u_{2}}^n)
	\leq \int_{\{u_1 \le  \phi -k\}}\mu.
\end{equation}
By using \eqref{eq2 le-lientuccungkididwedgedc}, \eqref{eq3 le-lientuccungkididwedgedc} and 
the fact $\lim_{k\to\infty}\int_{\{u_1\le \phi -k\}}\mu=0$, we get \eqref{eq1 le-lientuccungkididwedgedc}.
Now, combining \eqref{eq0 le-lientuccungkididwedgedc} and \eqref{eq1 le-lientuccungkididwedgedc}, we obtain 
$$	\int_X  -\chi (u_{1}-u_{2})  (\theta_{\psi_{1}}^n- \theta_{\psi_{2}}^n) \le 
C_2\varrho B^2(1-\tilde{\chi}(-1))^2 Q^{\circ (n)}(I^0_\chi(u_{1}, u_{2})).$$
The proof is completed.
\endproof

\begin{lemma}\label{le-secondhieuu1u2} 
	Let $M\geq 1$ and 	$\chi, \tilde{\chi}\in\widetilde{\mathcal{W}}^-\cup\mathcal{W}^{+}_M$
	such that $\tilde{\chi}\leq\chi$ and $\chi\in \Cc^1(\R)$.
Let $u_1, u_2, ..., u_{n+2} \in \mathcal{E}(X, \theta, \phi)$ such that $u_1 \le u_2$ and
	$u_j-\phi$ is bounded ($j=1, 2, ..., n+2$), where $\phi$ is a negative $\theta$-psh function satisfying 
	$\varrho:=\vol(\theta_\phi)>0$. 
 Denote 
 $$T=\theta_{u_4}\wedge...\wedge\theta_{u_{n+2}}, \quad I= \left|\int_X\chi'(u_1-u_2)d(u_1-u_2)\wedge d^c(u_1-u_3)\wedge T\right|,$$
 and
 $$J=\int_X\chi'(u_1-u_2)d(u_1-u_2)\wedge d^c(u_1-u_2)\wedge T.$$
	Then there exists $C>0$ depending only on $n$ and $M$ such that
	$$I \leq C\varrho B(1-\tilde{\chi}(-1))
	Q(J/\varrho),$$
 where $B:=\sum_{j=1}^{n+2} \max\{E^0_{\tilde{\chi},\theta, \phi}(u_j),1 \}$
	and $Q$ is defined by \eqref{eq QB}. 
\end{lemma}

Clearly if $\chi, \tilde{\chi} \in \widetilde{\mathcal{W}}^-$, then the above constant $C$ does not depend on $M$. 

\proof 
In this proof, we use the symbols $\lesssim$ and $\gtrsim$ for inequalities modulo a  constant depending only on $n$ and $M$.
By Theorem \ref{th-integrabypart} and Lemma \ref{le-sosanhnangnluongintegrabig}, we have
$$I=\left|\int_X -\chi(u_1-u_2)dd^c(u_1-u_3)\wedge T\right|\lesssim \varrho B= \varrho B Q(1).$$ Therefore, without loss of generality,
we can assume that $J/\varrho<1$.
Approximating $u_3$ by $u_3-\delta$ with $\delta\searrow 0$, we can assume that $u_3<\phi$ on $X$.

For each $0<\epsilon<1/2$ we denote
\begin{center}
	$U(\epsilon)=\{u_1-u_2<\epsilon(u_1+u_3-2\phi)\}, V(\epsilon)=\{u_1-u_2>\epsilon(u_1+u_3-2\phi)\},$
\end{center}
	and $\Gamma(\epsilon)=\{u_1-u_2=\epsilon(u_1+u_3-2\phi)\}$.
Since $\Gamma(\epsilon_1)\cap\Gamma(\epsilon_2)=\emptyset$ for every $\epsilon_1\neq\epsilon_2$ (note $u_3<\phi$),
we have
\begin{equation}\label{eq epsilon lemforpr3.5}
	\int_{\Gamma(\epsilon)}d(u_1-u_3)\wedge d^c(u_1-u_3)\wedge T=0,
\end{equation}
for almost everywhere $\epsilon\in (0, 1/2)$. 

Let $0<\epsilon<1/2$ be a constant  satisfying \eqref{eq epsilon lemforpr3.5}. To simplify the notation, from now on, we write $U,V, \Gamma$ for $U(\epsilon), V(\epsilon), \Gamma(\epsilon)$ respectively.
Denote 
$$\tilde{u}_1=\dfrac{u_1+\epsilon u_3}{1+\epsilon},\qquad
\tilde{u}_2=\max\left\{\dfrac{u_2+\epsilon u_3}{1+\epsilon},\quad \dfrac{(1-\epsilon)u_1+2\epsilon\phi}{1+\epsilon}\right\} \quad\mbox{and}\quad \tilde{\varphi}=\tilde{u}_1-\tilde{u}_2.$$
Then
$\varphi:=(u_1-u_2)=(1+\epsilon)\tilde{\varphi}$ on $U$.
Hence
\begin{align*}
	I&=\left|\int_X-\chi(\varphi)dd^c(u_1-u_3)\wedge T\right|\\
	&\leq \left|\int_U-\chi(\varphi)dd^c(u_1-u_3)\wedge T\right|
	+\left|\int_{X\setminus U}-\chi(\varphi)dd^c(u_1-u_3)\wedge T\right|\\
	&\leq \left|\int_U-\chi((1+\epsilon)\tilde{\varphi})dd^c(u_1-u_3)\wedge T\right|
	+\left|\int_{X\setminus U}-\chi(\varphi)(\theta_{u_1}+\theta_{u_3})\wedge T\right|\\
	&\leq \left|\int_U-\chi((1+\epsilon)\tilde{\varphi})dd^c(u_1-u_3)\wedge T\right|
	+\left|\int_{X\setminus U}-\chi(\epsilon(u_1+u_3-2\phi))(\theta_{u_1}+\theta_{u_3})\wedge T\right|\\
	&:=I_1+I_2,
\end{align*}
where in the last inequality we used the fact that $\chi$ is increasing and $\varphi \ge \epsilon(u_1+ u_2- 2 \phi)$ on $X \backslash U$. By Lemma \ref{le-sosanhnangnluongintegrabig}, we have
$E^0_{\tilde{\chi},\theta, \phi}\left(\dfrac{u_1+u_3}{2}\right)\lesssim B$. Therefore, it follows from Lemma
\ref{le-uocluongchiepsilon} that
\begin{equation}\label{eqI2 lemforpr3.5}
	I_2\leq 2\int_{X}-\chi\left(2\epsilon\left(\dfrac{u_1+u_3}{2}-\phi\right)\right)\theta_{(u_1+u_3)/2}\wedge T \lesssim B\varrho(1-\tilde{\chi}(-1))Q_{0}(2\epsilon).
\end{equation}
In order to estimate $I_1$, we divide it into two terms
\begin{align*}
	I_1&\leq  \left|\int_X-\chi((1+\epsilon)\tilde{\varphi})dd^c(u_1-u_3)\wedge T\right|
	+ \left|\int_{X\setminus U}-\chi((1+\epsilon)\tilde{\varphi})dd^c(u_1-u_3)\wedge T\right|\\
	&:=I_3+I_4.
\end{align*}

Note that $\tilde{u}_1-\tilde{u}_2=\epsilon(u_1+u_3-2\phi)/(1+\epsilon)$ on $X\setminus U$. Hence
$$I_4\leq \int_{X\setminus U}-\chi((1+\epsilon)\tilde{\varphi})(\theta_{u_1}+\theta_{u_3})\wedge T
\leq \int_{X\setminus U}-\chi(\epsilon(u_1+u_2-2\phi))(\theta_{u_1}+\theta_{u_3})\wedge T.$$
Using Lemma \ref{le-uocluongchiepsilon} again, we get
\begin{equation}\label{eqI4 lemforpr3.5}
	I_4\lesssim B\varrho(1-\tilde{\chi}(-1))Q_{0}(2\epsilon).
\end{equation}
Using integration by parts, we have
	$$I_3=(1+\epsilon)\left|\int_X\chi'((1+\epsilon)\tilde{\varphi})d\tilde{\varphi}\wedge
	d^c(u_1-u_3)\wedge T\right|.$$
Moreover, by Cauchy-Schwarz inequality and by the choice of $\epsilon$ (see \eqref{eq epsilon lemforpr3.5}), we get
$$\int_{\Gamma}\chi'((1+\epsilon)\tilde{\varphi})d\tilde{\varphi}\wedge
	d^c(u_1-u_3)\wedge T=0.$$
Hence
\begin{equation}\label{eqI3 lemforpr3.5}
	I_3=(1+\epsilon)\left|\int_{U\cup V}\chi'((1+\epsilon)\tilde{\varphi})d\tilde{\varphi}\wedge
	d^c(u_1-u_3)\wedge T\right|\leq (1+\epsilon)(I_5I_6)^{1/2}
\end{equation}
where
$$I_5=\int_{U\cup V}\chi'((1+\epsilon)\tilde{\varphi})d(u_1-u_3)\wedge
d^c(u_1-u_3)\wedge T,$$
and
$$I_6=\int_{U\cup V}\chi'((1+\epsilon)\tilde{\varphi})d\tilde{\varphi}\wedge
d^c\tilde{\varphi}\wedge T.$$
Since $(1+\epsilon)\tilde{\varphi}\leq \epsilon(u_1+u_3-2\phi)$, if $\chi\in\widetilde{\mathcal{W}}^-$ (hence $\chi'$ is nonnegative and increasing on $\R_{ \le 0}$) then
\begin{align*}
	I_5&\leq \int_X\chi'(\epsilon(u_1+u_3-2\phi))d(u_1-u_3)\wedge
	d^c(u_1-u_3)\wedge T\\
	& \lesssim \int_X\chi'(\epsilon(u_1+u_3-2\phi))d(u_1-\phi)\wedge d^c(u_1-\phi)\wedge T\\
	&+\int_X\chi'(\epsilon(u_1+u_3-2\phi))d(u_3-\phi)\wedge d^c(u_3-\phi)\wedge T\\
	&\leq \int_X\chi'(\epsilon(u_1-\phi))d(u_1-\phi)\wedge	d^c(u_1-\phi)\wedge T
	+\int_X\chi'(\epsilon(u_3-\phi))d (u_3-\phi)\wedge d^c(u_3-\phi)\wedge T\\
	&=\epsilon^{-1}\int_X-\chi(\epsilon(u_1-\phi))dd^c(u_1-\phi)\wedge T
	+\epsilon^{-1}\int_X-\chi(\epsilon(u_3-\phi))dd^c(u_3-\phi)\wedge T\\
	&\lesssim B\varrho(1-\tilde{\chi}(-1))\epsilon^{-1}Q_{0}(\epsilon),
\end{align*}
where the last estimate holds due to Lemma \ref{le-uocluongchiepsilon}.

Denote $v_1:=(u_1+2u_3)/3$ and $v_2:=(2u_1+u_3)/3$.
Since 
$$(1+\epsilon)(\tilde{u}_1-\tilde{u}_2)\geq u_1+u_3-2\phi, \quad u_1- u_3= -3(v_1- v_2),$$
one sees that  if $\chi\in\mathcal{W}_M^+$ (hence $\chi'$ is nonnegative and decreasing in $\R_{\le 0}$) then
\begin{align*}
	I_5&\leq \int_X\chi'((u_1+u_3-2\phi))d(u_1-u_3)\wedge
	d^c(u_1-u_3)\wedge T\\
	&\lesssim \int_X\chi'((u_1+u_3-2\phi))\big(d(v_1-\phi)\wedge
	d^c(v_1-\phi)+ d(v_2-\phi)\wedge d^c(v_2-\phi)\big)\wedge T\\
	&\leq \int_X\chi'(3(v_1-\phi))d(v_1-\phi)\wedge
	d^c(v_1-\phi)\wedge T
	+ \int_X\chi'(3(v_2-\phi))d(v_2-\phi)\wedge
	d^c(v_2-\phi)\wedge T\\
	&= \dfrac{1}{3}\int_X-\chi(3(v_1-\phi))dd^c(v_1-\phi)\wedge T
	+ \dfrac{1}{3}\int_X-\chi(3(v_2-\phi))dd^c(v_2-\phi)\wedge T\\
	&\leq \int_X-\chi(3(v_1-\phi))(\theta_{v_1}+\theta_{\phi})\wedge T
	+ \int_X-\chi(3(v_2-\phi))(\theta_{v_2}+\theta_{\phi})\wedge T\\
	&\leq 3^{M}\int_X-\chi(v_1-\phi)(\theta_{v_1}+\theta_{\phi})\wedge T
	+ 3^{M}\int_X-\chi(v_2-\phi)(\theta_{v_2}+\theta_{\phi})\wedge T\\
	&\lesssim B\varrho,
\end{align*}
where the two last estimates hold due to Lemma \ref{le-sosanhnangnluongintegrabig} and the fact
$$\log(-\chi(3t))-\log(-\chi(t))=\int_t^{3t}\dfrac{\chi'(s)}{\chi(s)}ds\le \int_t^{3t}\dfrac{M}{s}ds
=M\log 3,$$
for every $\chi\in\mathcal{W}_M^+$ and $t \le 0$. Combining the estimates in both cases, we obtain
\begin{equation}\label{eqI5 lemforpr3.5}
	I_5\lesssim B\varrho (1-\tilde{\chi}(-1))\dfrac{Q(\epsilon)^2}{\epsilon},
\end{equation}
where we used the inequality $Q(\epsilon) \ge \epsilon^{1/2}$ if $\chi \in \widetilde{W}^+_M$.   Now, we estimate $I_6$. Since $U, V$ are open in the plurifine topology and 
$$(1+\epsilon)\tilde{\varphi}=\begin{cases}\varphi\quad\mbox{on}\quad U\\
\epsilon(u_1+u_3-2\phi)\quad\mbox{on}\quad V\end{cases},$$ we have
\begin{align*}
	I_6&=\int_{U}\chi'((1+\epsilon)\tilde{\varphi})d\tilde{\varphi}\wedge
	d^c\tilde{\varphi}\wedge T+\int_{V}\chi'((1+\epsilon)\tilde{\varphi})d\tilde{\varphi}\wedge
	d^c\tilde{\varphi}\wedge T\\
	&=(1+\epsilon)^{-2}\int_{U}\chi'(\varphi)d\varphi\wedge
	d^c\varphi\wedge T\\
	&+\dfrac{\epsilon^2}{(1+\epsilon)^2}\int_{V}\chi'(\epsilon(u_1+u_3-2\phi))d(u_1+u_3-2\phi)\wedge
	d^c(u_1+u_3-2\phi)\wedge T\\
		&\leq J+\epsilon^2\int_{X}\chi'(\epsilon(u_1+u_3-2\phi))d(u_1+u_3-2\phi)\wedge
	d^c(u_1+u_3-2\phi)\wedge T\\
		&= J+\epsilon\int_{X}-\chi(\epsilon(u_1+u_3-2\phi))dd^c(u_1+u_3-2\phi)\wedge T.
\end{align*}
Therefore, it follows from Lemma \ref{le-uocluongchiepsilon} that
\begin{equation}\label{eqI6 lemforpr3.5}
	I_6\lesssim J+B\varrho (1-\tilde{\chi}(-1))\epsilon Q_{0}(2\epsilon).
\end{equation}
Combining \eqref{eqI2 lemforpr3.5}, \eqref{eqI3 lemforpr3.5}, \eqref{eqI4 lemforpr3.5}, \eqref{eqI5 lemforpr3.5}
and \eqref{eqI6 lemforpr3.5}, we get
\begin{align*}
	I\leq I_1+I_2&\leq I_3+I_4+I_2\\
	&\lesssim (I_5I_6)^{1/2}+I_4+I_2\\
	&\lesssim \left( B\varrho (1-\tilde{\chi}(-1))\epsilon^{-1}J\right)^{1/2} Q(\epsilon)+ B\varrho (1-\tilde{\chi}(-1)) Q(2\epsilon)^2.
\end{align*}
Letting $\epsilon\searrow J/(2\varrho)$ (and $\epsilon$ satisfies \eqref{eq epsilon lemforpr3.5}), we obtain
 $$I\lesssim B\varrho (1-\tilde{\chi}(-1)) Q(J/\varrho).$$
 The proof is completed.
\endproof

\begin{proposition} \label{pro-mainstabilitylowenergyconvex} Let $\chi, \tilde{\chi} \in\widetilde{\mathcal{W}}^-\cup\mathcal{W}^{+}_M$	such that $\tilde{\chi}\leq\chi$ and $\chi\in \Cc^1(\R)$.  Let $u_1, u_2, u_3 \in \mathcal{E}(X, \theta, \phi)$ such that $u_1 \le u_2$ and
	$u_j-\phi$ is bounded ($j=1, 2, 3$), where $\phi$ is a negative $\theta$-psh function satisfying 
	$\varrho:=\vol(\theta_\phi)>0$.  Then there exists a constant $C_n>0$ depending only on $n$ and $M$ such that 
	\begin{align} \label{ine-chigradientfixedtypelower}
		\int_X  \chi'(u_1-u_2) d(u_1- u_2) \wedge \dc (u_1- u_2) \wedge \theta_{u_3}^{n-1}  \le C_n\varrho B^2(1-\tilde{\chi}(-1))^2 Q^{\circ (n-1)}\big(I^0_\chi(u_1,u_2)\big),
	\end{align}
	where $B:=\sum_{j=1}^3 \max\{E^0_{\tilde{\chi},\theta, \phi}(u_j),1 \}$
	and $Q$ is defined by \eqref{eq QB}. 
\end{proposition}

\begin{proof} 
 Let 
$$\varphi:= u_1- u_2, \quad T:= \sum_{j=1}^{n-1} \theta_{u_1}^j \wedge \theta_{u_2}^{n-1-j} ,$$
and
$$T_{k,l}:= \theta_{u_1}^{k} \wedge \theta_{u_2}^l   \wedge  \theta_{u_3}^{n-k-l-1}, \quad L_{k,l}:= \int_X \chi'(\varphi) d \varphi \wedge \dc \varphi \wedge T_{k,l}.$$
Observe 
$$\theta_{u_2}^n- \theta_{u_1}^n=-  \ddc \varphi \wedge T$$
 and 
\begin{align}\label{ine-Tklntru1}
L_{k,n-1-k} \le \int_X \chi'(\varphi) d \varphi \wedge \dc \varphi \wedge T = \varrho I^0_\chi(u_1, u_2)
 \end{align}
 by integration by parts. We now prove by inverse induction on $m:= k+l$ that 
 \begin{align} \label{ine-chigradientfixedtypeLkl}
L_{k,l}  \le   C_{m,n}\varrho B^2(1-\tilde{\chi}(-1))^2Q^{\circ (n-1-k-l)}\big(I^0_\chi(u_1,u_2)\big),
\end{align}
for some  constant $C_{m,n}>1$ depending only on $m, n$ and $M$.  
The desired assertion (\ref{ine-chigradientfixedtypelower}) is the case where $k=l=0$. 
In what follows  we use the symbols $\lesssim$ and $\gtrsim$ for inequalities modulo a constant depending
only on $n$ and $M$. 
We have checked  (\ref{ine-chigradientfixedtypeLkl}) for $k+l=n-1$.  Suppose that (\ref{ine-chigradientfixedtypeLkl}) holds for $k+l=m$ with $0<m\leq n-1$.
 We will verify it for $L_{k-1,l}$, where
$k+l=m$ and $k>1$. The case $L_{k,l-1}$ is done similarly.  

Denote $S_{k-1, l}=\theta_{u_1}^{k-1} \wedge \theta_{u_2}^l   \wedge  \theta_{u_3}^{n-k-l-1}$. Then
$$L_{k-1, l}-L_{k, l}=\int_X \chi'(\varphi) d \varphi \wedge \dc \varphi\wedge dd^c(u_3-u_1) \wedge S_{k-1,l}.$$
Using integration by parts, we have
\begin{align*}
	L_{k-1, l}-L_{k, l}
	&=\int_X -\chi(\varphi)dd^c(\varphi)\wedge dd^c(u_3-u_1) \wedge S_{k-1,l}\\
	&=\int_X -\chi(\varphi) dd^c(u_3-u_1) \wedge T_{k,l}-\int_X -\chi(\varphi) dd^c(u_3-u_1) \wedge T_{k-1,l+1}\\
&=\int_X\chi'(\varphi) d\varphi\wedge d^c(u_3-u_1) \wedge T_{k,l}
-\int_X\chi'(\varphi) d\varphi\wedge d^c(u_3-u_1) \wedge T_{k-1,l+1}\\
\end{align*}
Therefore, it follows from Lemma \ref{le-secondhieuu1u2} that
$$L_{k-1, l}-L_{k, l}\lesssim \varrho B(1-\tilde{\chi}(-1))\left(Q(L_{k, l}/\varrho)+Q(L_{k-1, l+1}/\varrho)\right).$$
Hence, by using the inductive hypothesis, we get
\begin{align*}
	L_{k-1, l}&\lesssim \varrho B^2(1-\tilde{\chi}(-1))^2Q^{\circ (n-1-m)}\big(I^0_\chi(u_1,u_2)\big)\\
	&+\varrho B(1-\tilde{\chi}(-1))Q\left(C_{m,n}B^2(1-\tilde{\chi}(-1))^2Q^{\circ (n-1-m)}\big(I^0_\chi(u_1,u_2)\big)\right)\\
	&\lesssim \varrho B^2(1-\tilde{\chi}(-1))^2Q^{\circ (n-m)}\big(I^0_\chi(u_1,u_2)\big).
\end{align*}
Here we use the fact $Q(t_1)\leq (t_1/t_2)^{1/2}Q(t_2)$ for every $t_1>t_2>0$ (see Lemma \ref{lem qt1t2}).

Thus, \eqref{ine-chigradientfixedtypeLkl} holds for $L_{k-1, l}$. This finishes the proof.
\end{proof}

\begin{lemma}\label{lem qt1t2}
	The function  $h(t)=\dfrac{(Q(t))^2}{t}$ is non-increasing in $\R_{>0}$.
\end{lemma}

\begin{proof}
	If $\chi\in \mathcal{W}_M^+$ then we have
			$$h(t)=
	\begin{cases}
		\dfrac{1}{t} \quad\mbox{if}\quad t\ge 1,\\
		1\quad\mbox{if}\quad 0<t<1,
	\end{cases}$$
is a non-increasing function.

	 We consider the case $\chi\in\widetilde{\mathcal{W}}^-$. We have
	 $$h(t)=
	 \begin{cases}
	 	\dfrac{1}{t} \quad\mbox{if}\quad t\ge 1,\\
	 	\dfrac{Q_0(t)}{tQ_0(1)} \quad\mbox{if}\quad 0<t<1.
	 \end{cases}$$
 It is clear that $h$ is decreasing in $[1, \infty)$.
 We need to show that $h$ is non-increasing in $(0, 1)$. Since $\chi$ is convex, we have
 $$\dfrac{\chi (t_1s)}{t_1s}\leq\dfrac{\chi(t_2s)}{t_2s},$$
 for every $0<t_2<t_1<1$ and $s<0$. Dividing both sides of the last estimate by $\tilde{\chi}(s)/s$, we get
 $$\dfrac{\chi (t_1s)}{t_1\tilde{\chi}(s)}\leq\dfrac{\chi(t_2s)}{t_2\tilde{\chi}(s)}.$$
 Taking the supremum of both sides, we obtain 
 $$\dfrac{Q_0(t_1)}{t_1}=\sup_{s\leq -1}\dfrac{\chi (t_1s)}{t_1\tilde{\chi}(s)}\leq\sup_{s\leq-1}\dfrac{\chi(t_2s)}{t_2\tilde{\chi}(s)}=
 \dfrac{Q_0(t_2)}{t_2}.$$
 Then $h(t_1)\leq h(t_2)$. Hence, $h$ is non-increasing in $(0, 1)$. The proof is completed.
\end{proof}

\begin{proof}[End of the proof of Theorem \ref{th-lowerenergy}] By Lemma \ref{le-lientuccungkididwedgedc}, we can assume that $u_j, \psi_j$ are of the same singularity type as $\phi$.  Now let $(\chi_j)_{j \in \N}$, $(\tilde{\chi}_j)_{j \in \N}$ be the sequence approximating $\chi$, $\tilde{\chi}$ respectively  in  
Lemma \ref{le-regularizedchi}. By Lebesgue's dominated convergence theorem, observe that 
$$\lim_{j \to \infty} I^0_{\chi_j}(u_1,u_2)= I^0_\chi(u_1,u_2)$$
and 
$$\lim_{j \to \infty} \int_X -\chi_j(u_1-u_2)(\theta_{\psi_1}^n- \theta_{\psi_2}^n)= \int_X -\chi(u_1-u_2)(\theta_{\psi_1}^n- \theta_{\psi_2}^n).$$
On the other hand, for $\epsilon \in (0,1]$ we also get
$$\lim_{j \to\infty} Q_{\chi_j, \tilde{\chi}_j}(\epsilon)=Q_{\chi, \tilde{\chi}}(\epsilon),$$
because for $t \le -1$,  one has 
$$\frac{\chi_j(\epsilon t)}{\tilde{\chi}_j(t)} \le \frac{\chi_j(\epsilon t)}{\tilde{\chi}(t)} \le \frac{\chi(\epsilon t)- 2^{-j}}{\tilde{\chi}(t)}$$
and
$$\frac{\chi_j(\epsilon t)}{\tilde{\chi}_j(t)} \ge \frac{\chi_j(\epsilon t)}{\tilde{\chi}(t)- 2^{-j}} $$
which converges to $\chi(\epsilon t)/\tilde{\chi}(t)$
(by Lemma \ref{le-regularizedchi}). Hence, by considering $\chi_j, \tilde{\chi}_j$ instead of $\chi, \tilde{\chi}$, we can further assume that $\chi \in \Cc^1(\R)$. 
 
  Let $L$ be the left-hand side of the desired inequality.  We have
 \begin{align*}
 	L&=\int_X-\chi(u_1-u_2)(\theta_{\psi_1}^n-\theta_{u_1}^n)
 	-\int_X-\chi(u_1-u_2)(\theta_{\psi_2}^n-\theta_{u_1}^n)\\
 	&=\int_X-\chi(u_1-u_2)dd^c(\psi_1-u_1)\wedge T_1-\int_X-\chi(u_1-u_2)dd^c(\psi_2-u_1)\wedge T_2\\
 	&=L_1-L_2,
 \end{align*}
 where $T_j=\sum_{l=0}^{n-1}\theta_{\psi_j}^l\wedge\theta_{u_1}^{n-l-1}$. Using integration by parts
 and Lemma \ref{le-secondhieuu1u2}, we get
 \begin{align*}
 	L_1&=\int_X\chi'(u_1-u_2)d(u_1-u_2)\wedge d^c(\psi_1-u_1)\wedge T_1\\
 	&\leq C_1\varrho B (1-\tilde{\chi}(-1))Q\left(\varrho^{-1}\int_X\chi'(u_1-u_2)d(u_1-u_2)\wedge d^c(u_1-u_2)\wedge T_1\right),
 \end{align*}
where $C_1>0$ depends only on $n$ and $M$. Observe that there is a dimensional constant $C'_1$ such that 
$$T_1 \le C_1'\, \theta_{(\psi_1+ u_1)/2}^{n-1}.$$ 
Moreover  one has
$$E_{\tilde{\chi},\theta,\phi}\big((\psi_1+ u_1)/2\big) \lesssim E_{\tilde{\chi},\theta,\phi}(\psi_1)+E_{\tilde{\chi},\theta,\phi}(u_1)$$
by Lemma \ref{le-sosanhnangnluongintegrabig}. Hence, it follows from Proposition \ref{pro-mainstabilitylowenergyconvex} (applied to $u_3:= (\psi_1+ u_1)/2$) that
 $$\varrho^{-1}\int_X\chi'(u_1-u_2)d(u_1-u_2)\wedge d^c(u_1-u_2)\wedge T_1\leq
 C_2 B^2(1-\tilde{\chi}(-1))^2Q^{\circ (n-1)}\left(I^0_{\chi}(u_1, u_2)\right),$$
 where $C_2>1$ depends only on $n$ and $M$.
 Then
 $$L_1\leq C_3\varrho B^2 (1-\tilde{\chi}(-1))^2Q^{\circ n}\left(I^0_{\chi}(u_1, u_2)\right),$$
 where $C_3>0$ depends only on $n$ and $M$. Here we use the fact $Q(t_1)\leq (t_1/t_2)^{1/2}Q(t_2)$ for every $t_1>t_2>0$ (Lemma \ref{lem qt1t2}).
 
  By the same arguments, we also have
 $$-L_2\leq C_4\varrho B^2 (1-\tilde{\chi}(-1))^2Q^{\circ n}\left(I^0_{\chi}(u_1, u_2)\right),$$
 where $C_4>0$ depends only on $n$ and $M$.
 
 Hence
 $$L=L_1-L_2\leq (C_3+C_4)\varrho B^2 (1-\tilde{\chi}(-1))^2Q^{\circ n}\left(I^0_{\chi}(u_1, u_2)\right).$$
   The proof is completed.  
\end{proof}

\subsection{Proof of Theorem \ref{th-lowerenergy-kocochuanhoa}}

Recall that for every Borel set $E$ in $X$, we define
$$\capK_{\theta, \phi}(E):= \sup \bigg\{\int_E \theta_h^n: \quad h \in \PSH(X,\theta), \quad   \phi-1 \le h \le \phi  \bigg\}.$$
The following is an improvement of results from \cite{Lu-Darvas-DiNezza-logconcave,Lu-Darvas-DiNezza-mono}  (see also \cite{BEGZ,Kolodziej_2003}).

\begin{theorem}\label{the P[u]-C<u}
	Let $A\geq 1$ be a constant and let $\theta$ be a closed smooth real $(1,1)$-form such that $\theta\leq A\omega$.
	Let $\phi\in \PSH(X, \theta)$
	and $0\leq f\in L^p(X)$  for some constant $p>1$
	such that $\phi=P[\phi]$ and $0<\int_X f \omega^n=\int_X\theta_{\phi}^n:=\varrho$.
	Assume $u\in\mathcal{E}(X, \theta, \phi)$ satisfies $\sup_X(u-\phi)=0$ and $\theta_u^n=f\omega^n.$
	Then, there exists a constant $C \ge 1$ depending only on $X, \omega, n$ and $p$ such that
	\begin{equation}
		u\geq\phi-C\, A\left(\log\|f\,\mathrm{vol}_{\omega}(X)^q/\varrho\|_{L^p}+\log A+1\right),
	\end{equation}
 where $\vol_\omega(X):= \int_X\omega^n$ and $q=\dfrac{p}{p-1}$.
\end{theorem}

By H\"older inequalities, one sees that
$$1=\int_X\dfrac{f}{\varrho}\omega^n\leq \|f/\varrho\|_{L^p}\left(\mathrm{vol}_{\omega}(X)\right)^q,$$
and then $\log\|f \vol_{\omega}(X)^q/\varrho\|_{L^p}\geq 0$.

\begin{proof}
	Without loss of generality, we can assume that $\vol_{\omega}(X)=1$.
Recall that there exists a constant $\nu>0$ depending only on $X, \omega$ such that 		
$$\int_X \exp\left(-\psi/\nu\right)\omega^n\leq C_0^2,$$
	for every $\psi\in\PSH(X, \omega)$ with $\sup_X\psi=0$, where
	$C_0 \ge 1$ is a constant depending only on $X$ and $\omega$. Consequently, one gets
	$$\int_X \exp\left(-\psi/(A\nu)\right)\omega^n\leq C_0^2,$$
	for every $\psi\in\PSH(X, \theta)\subset\PSH(X, A\omega)$ with $\sup_X\psi=0$.
	By the same arguments as in the proof of
	\cite[Proposition 4.30]{Lu-Darvas-DiNezza-mono} (use  \cite[Lemma 3.9]{Lu-Darvas-DiNezza-logconcave}
	instead of \cite[Lemma 4.9]{Lu-Darvas-DiNezza-mono}), we have 
	$$\int_E\omega^n\leq C^2_0\exp\left(-\dfrac{1}{2A\nu}\left(\dfrac{\capK_{\theta, \phi}(E)}{\varrho}\right)^{-1/n}\right),$$
	for every Borel set $E\subset X$. Therefore, by the H\"older inequality and the fact
	$e^{-1/t}\leq m!t^m$ for every $m \in \N$ and every $t >0$, there exists $A_0>0$ depending only on 
	$X, \omega, n$ and $p$ such that
	\begin{equation}\label{eq0.1 proof the P[u]-C<u}
		\varrho^{-1}\int_E\theta_u^n= \int_E (f/\varrho) \omega^n \leq \|f/\varrho\|_{L^p}\left(\int_E \omega^n\right)^{1/q}\leq
		 A_0A^{2n}\|f/\varrho\|_{L^p}
		\dfrac{\capK_{\theta, \phi}(E)^2}{\varrho^2},
	\end{equation}
	for every Borel set $E\subset X$, where $1/p+1/q=1$. On the other hand, denoting
	$b=(A\nu q)^{-1}$ and $B_0=(C_0^2)^{1/q}$, we have
	\begin{equation}\label{eq0.2 proof the P[u]-C<u}
		\varrho^{-1}\int_Xe^{-b w}\theta_u^n\leq \|f/\varrho\|_{L^p}\left(\int_Xe^{-b q w}\omega^n\right)^{1/q}
		\leq B_0\|f/\varrho\|_{L^p},
	\end{equation}
	for every $w\in \PSH(X, \theta)$ with $\sup_X w=0$. 
	
	For every $h\in \PSH(X, \theta)$ with $\phi-1\leq h\leq \phi$, for each $0\leq t\leq 1$ and $s>0$, we have
	\begin{align*}
		t^n\int\limits_{\{u<\phi-t-s\}}\theta_h^n\leq\int\limits_{\{u<(1-t)\phi+t h-s\}}\theta_{(1-t)\phi+t h}^n
		&\leq \int\limits_{\{u<(1-t)\phi+t h-s\}}\theta_{u}^n\\
		&\leq \int\limits_{\{u<\phi-s\}}\theta_{u}^n,
	\end{align*}
	where the third estimate holds due to the comparison principle \cite[Lemma 2.3]{Lu-Darvas-DiNezza-logconcave}.
	Then
	\begin{equation}\label{eq1 proof the P[u]-C<u}
		t^n \, \capK_{\theta,\phi}(u<\phi-t-s)\leq \int\limits_{\{u<\phi-s\}}\theta_{u}^n,
	\end{equation}
	for every $0\leq t\leq 1$, $s>0$. Therefore, it follows from \eqref{eq0.1 proof the P[u]-C<u} that
	\begin{equation*} 
		t^n \, \varrho^{-1}\capK_{\theta,\phi}(u<\phi-t-s)\leq A_1 \,  \varrho^{-2}\capK_{\theta,\phi}(u<\phi-s)^{2},
	\end{equation*}
	where $A_1=A_0A^{2n}\|f/\varrho\|_{L^p}$.
	Putting $g(s)=\varrho^{-1/n}\capK_{\theta,\phi}(u<\phi-s)^{1/n}$, the above inequality becomes
	\begin{equation*} 
		tg(t+s)\leq A_1^{1/n}g(s)^2.
	\end{equation*}
	Hence, it follows from \cite[Lemma 2.4 and Remark 2.5]{EGZ} that if $g(s_0)<1/(2A_1^{1/n})$
	then $g(s)=0$ for all $s\geq s_0+2$. Moreover, by \eqref{eq1 proof the P[u]-C<u}
	and the condition \eqref{eq0.2 proof the P[u]-C<u}, we have
	$$g(s+1)^n\leq \varrho^{-1}\int\limits_{\{u<\phi-s\}}\theta_{u}^n\leq \varrho^{-1}\int\limits_{X}e^{b(\phi-u-s)}\theta_{u}^n
	\leq B_1 \, e^{-bs},$$
	for every $s>0$, where $B_1=B_0\|f/\varrho\|_{L^p}$. Then $g(s+1)<1/(2A_1^{1/n})$ provided that
	$$s>\frac{n\log 2+\log A_1}{b}+\frac{\log B_1}{b} \cdot$$ 
	Hence $g(s)=0$ for every 
	$$s\geq \frac{n\log 2+\log A_1}{b}+\frac{\log B_1}{b}+4.$$
	Thus
	\begin{center}
		$u\geq\phi -\left(\dfrac{n\log 2+\log A_1}{b}+\dfrac{\log B_1}{b}+4\right)
		= \phi-C_1\log\|f/\varrho\|_{L^p}-C_2,$
	\end{center}
	where $C_1=\frac{2}{b}=2\nu q A$ and 
	\begin{align*}
	C_2 &=4+\frac{n\log 2+\log A_0+\log B_0+2n\log A}{b}\\
	&=4+\nu q(n\log 2+\log A_0+\log B_0+2n\log A) A.
	\end{align*}
		The proof is finished.
\end{proof}

\begin{lemma}\label{lem vol estimate}
	There exists a constant $C>0$ depending only on $n, X$ and $\omega$ such that for every $u\in\PSH (X, \omega)$ satisfying $\sup_X u=0$ and for every constant $0<t \le 1$, one has
	\begin{equation}
		\int_{\{u>-t\}} \omega^n \geq C t^{2n}.
	\end{equation}
\end{lemma}

\begin{proof}
	Let $(U_j, \varphi_j)_{j=1}^m$ such that $U_j\subset X$ are open, $\varphi_j: 4\B \longrightarrow U_j$
	are biholomorphic and $\cup_{j=1}^m\varphi_j(\B)=X$ (where $\B$ is the open unit ball in $\C^n$), and  there is a smooth psh	function $\rho_j$ in $U_j$ such that $dd^c\rho_j=\omega$ for $1 \le j \le m$. Denote
	$$C_{\rho}=\sup\limits_{1\leq j\leq m}\sup\limits_{2\B}\|\nabla(\rho_j\circ\varphi_j)\|.$$
	Assume $u(z_0)=0$. Then there exists $1\leq j_0\leq m$ such that $z_0\in\varphi_{j_0}(\B)$.
	Denote $w_0=\varphi_{j_0}^{-1}(z_0)$, $\widehat{u}(w)=u\circ \varphi_{j_0}(w)$
	and $\widehat{\rho}(w)=\rho_{j_0}\circ\varphi_{j_0}(w)-\rho_{j_0}\circ\varphi_{j_0}(w_0)$. 
	By the plurisubharmonicity of $\widehat{u}+\widehat{\rho}$, for every $t>0$ and $0<r<1$,
	we have
	\begin{align*}
		0=(\widehat{u}+\widehat{\rho})(w_0)
		&\leq\dfrac{1}{\vol_{\C^n}(r\B)}\int_{r\B}(\widehat{u}+\widehat{\rho})dV_{2n}\\
		&\leq C_{\rho}r+\dfrac{1}{c_{2n}r^{2n}}\int_{r\B}\widehat{u} dV_{2n}\\
		&\leq C_{\rho}r-\dfrac{t}{c_{2n}r^{2n}}\int_{r\B \cap\{\widehat{u}\leq-t\}} dV_{2n}\\
		&\leq C_{\rho}r-t+\dfrac{t}{c_{2n}r^{2n}}\int_{r\B \cap\{\widehat{u}>-t\}}dV_{2n}\\
		&\leq C_{\rho}r-t+\dfrac{C_{\omega}t}{r^{2n}}\mathrm{vol}_\omega(\{u>-t\}),\\
	\end{align*}
	where $c_{2n}=\mathrm{vol}_{\C^n}(\B)$ and $C_{\omega}>0$ is a constant depending only on $n, X, \omega$.
	It follows that
	$$\mathrm{vol}_\omega(\{u>-t\})
	\geq \dfrac{r^{2n}}{C_{\omega}}\left(1-\dfrac{C_{\rho}r}{t}\right).$$
	Hence, for every $0<t<1$, by choosing $r=\frac{t}{1+C_{\rho}}$, we have
	$$\mathrm{vol}_\omega(\{u>-t\})\geq C t^{2n},$$
	where $C=\dfrac{1}{C_{\omega}(1+C_{\rho})^{2n+1}}$ depends only on $n, X$ and $\omega$.
\end{proof}

\begin{proof}[End of the proof of Theorem \ref{th-lowerenergy-kocochuanhoa}]
		Without loss of generality, we can assume that 
	$u_1\leq u_2$. Denote 
	$W_t=\{u_1>a_1-t\}$ for $0<t\leq 1$. We have
	\begin{equation}\label{eq1 proof GZ12}
		\int_{W_t}-\chi(u_1-u_2)\omega^n\leq\int_{W_t}-\chi(u_1-a_2)\omega^n\leq -b_t\chi(a_1-a_2-t),
	\end{equation} 
where $	b_t:=\vol (W_t)$.

	It follows from Lemma \ref{lem vol estimate} that $W_t\neq\emptyset$. Moreover,
	\begin{equation}\label{eq1.1 proof GZ12}
		b_t:=\int_{W_t}\omega^n\geq C_1\left(\dfrac{t}{A}\right)^{2n},
	\end{equation}
where $C_1>0$ is a constant depending only on $n, X$ and $\omega$. By \cite[Theorem A]{Lu-Darvas-DiNezza-logconcave}
(see also \cite[Theorem 3]{Vu_Do-MA}), there exists a unique $\varphi\in\mathcal{E}(X, \theta, \phi)$
with $\sup_X (\varphi-\phi)=0$ such that
$$\theta_{\varphi}^n=\dfrac{\varrho}{b_t}\mathbf{1}_{W_t}\omega^n.$$
It follows from Theorem \ref{the P[u]-C<u} that
	\begin{equation}\label{eq2 proof GZ12}
	\phi-C_2A\left(-\log t+\log A+1\right)\leq\varphi\leq 	\phi,
\end{equation}
for some constant $C_2 \ge 1$ depending only on $n, X$ and $\omega$. Thus, we have
	$$E^0_{\tilde{\chi}, \theta, \phi}(\varphi)\leq -\tilde{\chi}\big(-C_2\, A\left(-\log t+\log A+1\right)\big)
	\leq -C_3\left(\log\dfrac{Ae}{t}\right)^M\tilde{\chi}(-A),$$
	where $C_3>0$ depends only on $n, X, \omega$ and $M$.
	
Hence, it follows from Theorem \ref{th-lowerenergy} that
\begin{equation}\label{eq3 proof GZ12}
	\int_X-\chi(u_1-u_2)(\theta_{\psi}^n-\theta_{\varphi}^n)\leq 
	C_4\varrho \left(\log\dfrac{Ae}{t}\right)^{2M}(B-\tilde{\chi}(-A))^2(1-\tilde{\chi}(-1))^2\lambda,
\end{equation}
where $\lambda=Q^{\circ (n)}(I^0_\chi(u_1,u_2))$ and $C_4>0$ depends only on $n, X, \omega$ and $M$.

Combining \eqref{eq1 proof GZ12} and \eqref{eq3 proof GZ12}, we get
$$\int_X-\chi(u_1-u_2)\theta_{\psi}^n\leq -\varrho\chi(a_1-a_2-t)
	+C_4\varrho \left(\log\dfrac{Ae}{t}\right)^{2M}(B-\tilde{\chi}(-A))^2(1-\tilde{\chi}(-1))^2 \lambda.$$
	Letting $t\rightarrow \lambda^{m}$, we get 
	\begin{align*}
		\int_X-\chi(u_1-u_2)\theta_{\psi}^n&\leq -\varrho\chi(a_1-a_2-\lambda^m)
		+C_4\varrho \left(\log\dfrac{Ae}{\lambda^m}\right)^{2M}(B-\tilde{\chi}(-A))^2(1-\tilde{\chi}(-1))^2 \lambda\\
		&\leq -\varrho\chi(a_1-a_2-\lambda^m)+
		 C_5\varrho \dfrac{A^{(1-\gamma)/m}}{\lambda^{1-\gamma}}(B-\tilde{\chi}(-A))^2(1-\tilde{\chi}(-1))^2 \lambda\\
			&\leq  -\varrho\chi(a_1-a_2-\lambda^m)+
			 C_5\varrho A^{(1-\gamma)/m} (B-\tilde{\chi}(-A))^2(1-\tilde{\chi}(-1))^2 \lambda^{\gamma},
	\end{align*}
	where $C_5>0$ depends only on $n, X, \omega, M, m$ and $\gamma$.	
		The proof is completed.
\end{proof}

\begin{remark}\label{re-nhohontru1} \normalfont The hypothesis that $\tilde{\chi} \le \chi$ in Theorems \ref{th-lowerenergy} and \ref{th-lowerenergy-kocochuanhoa} can be slightly relaxed: the same statement remains true if $\tilde{\chi} \le \chi$ on $(-\infty, -1]$ and $\chi(-1)=-1$. Indeed, we only need the inequality $\tilde{\chi} \le \chi$ to guarantee that $E_{\chi,\theta,\phi}(u) \le E_{\tilde{\chi}, \theta,\phi}(u)$ for $u \in \PSH(X, \theta,\phi)$. If we only have $\tilde{\chi} \le \chi$ on $(-\infty, -1]$, then there holds 
$$E_{\chi,\theta,\phi}(u) \le E_{\tilde{\chi}, \theta,\phi}(u)-\chi(-1)\vol(\theta_\phi).$$
 This is still sufficient for the proof of Theorems \ref{th-lowerenergy} and \ref{th-lowerenergy-kocochuanhoa}.  

Later we will apply Theorem \ref{th-lowerenergy-kocochuanhoa} to the special case where $\chi(t)= \max\{t, -1\}$ and $\tilde{\chi} \in \widetilde{\mathcal{W}}^-$ with $\tilde{\chi}(-1)=-1$. In this case, we can compute explicitly $Q_{0,\chi,\tilde{\chi}}(\epsilon)= \sup_{\{t \le  -1\}}\frac{\chi(\epsilon t)}{\tilde{\chi}(t)}$ as follows. Observe that 
$$Q_{0,\chi,\tilde{\chi}}(\epsilon)= \max\big\{ \sup_{-\epsilon^{-1} \le t \le -1} \frac{\chi(\epsilon t)}{\tilde{\chi}(t)}, \sup_{t \le - \epsilon^{-1}}\frac{\chi(\epsilon t)}{\tilde{\chi}(t)}  \big\}=\max\big\{ \sup_{-\epsilon^{-1} \le t \le -1} \frac{\epsilon t}{\tilde{\chi}(t)}, \frac{-1}{\tilde{\chi}(-\epsilon^{-1})}  \big\}.$$
Since $\tilde{\chi} \in \widetilde{\mathcal{W}}^-$, the function $t/\tilde{\chi}(t)$ is decreasing, hence, $Q_{0,\chi,\tilde{\chi}}(\epsilon)= \big(-\tilde{\chi}(-\epsilon^{-1})\big)^{-1}$. 

If $\chi(t)= \tilde{\chi}(t)= -(-t)^p$ for some constant $p>0$, then one sees directly that $Q_{0,\chi,\tilde{\chi}}(\epsilon)= \epsilon^p$. However we will not use this special case in applications.
\end{remark}

\subsection{A counter-example} \label{ex-chingabangchi} 

Let 	$\chi(t):=-\log (-t+1)\in \mathcal{W}^-$. In this subsection, to simplify the notation, we denote $E_\chi(u):= E_{\chi, \omega, 0}(u)$, where by $0$ we mean the constant function equal to $0$.    Our goal in this subsection is to construct sequences of functions
$u_m, v_m\in \PSH(X, \omega)\cap L^{\infty}(X)$ such that the following properties hold:

(i) $0\geq u_m\geq v_m$, $\sup_X u_m= \sup_X v_m =0$, 

(ii) $u_m,v_m  \to 0$ in $L^1$ as $m\rightarrow\infty$,

(iii)
 $\sup_m(E_{\chi}(u_m)+E_{\chi}(v_m))<\infty$ and
 $\lim_{m\to\infty}I_{\chi}(u_m, v_m)=0$ but

(iv)  $$\inf_m \int_X-\chi(u_m-v_m)(dd^cv_m+ \omega)^n>0.$$

As a consequence of our construction of $u_m,v_m$ below, we see that 
Theorem \ref{th-lowerenergy-phanintro} (and Theorem \ref{the-mainstabilitylowenergyconvexintro}) does not hold in general if $\chi= \tilde{\chi} \in \mathcal{W}^-$. Here is our construction. On the unit ball  $\mathbb{B}$ of $\mathbb{C}^n$, we define
	$$\varphi_m=\max\{\log|z|, -e^{m}\}\quad\mbox{and}\quad F_m=\{z\in\B: \log|z|=-e^m\},  \quad m>0.$$

\begin{lemma} We have
	\begin{equation}\label{1example}
	\int_{F_m}(dd^c\varphi_m)^k\wedge (dd^c|z|^2)^{n-k}=\begin{cases}
	O(e^{-e^m})\quad\mbox{if}\quad k<n,\\
	c\quad\mbox{if}\quad k=n,
	\end{cases}
	\end{equation}
	where $c:=\int_{\{z=0\}}(dd^c \log|z|)^n>0$.
	\end{lemma}
	
	\proof 
The case $k=n$ follows from Stokes' theorem.  
We consider now $k<n$. Let $\B_r$ be the ball of radius $r>0$ centered at $0$ in $\C^n$. Observe that $\varphi_m = \log |z|$ on an open neighborhood of $\partial \B_{2 e^{-e^m}}$. Using this and Stokes' theorem, we obtain
\begin{align*}
\int_{F_m}(dd^c\varphi_m)^k\wedge (dd^c|z|^2)^{n-k} & \le \int_{\B_{2 e^{-e^m}}}(dd^c\varphi_m)^k\wedge (dd^c|z|^2)^{n-k}\\
&=\int_{\B_{2 e^{-e^m}}}(dd^c\log |z|)^k\wedge (dd^c|z|^2)^{n-k}.
\end{align*}	 
By direct computations (and approximating $\log |z|$ by $\frac{1}{2}\log (|z|^2+\epsilon)$ as $\epsilon \to 0$), we see that 
$$\int_{\B_{2 e^{-e^m}}}(dd^c\log |z|)^k\wedge (dd^c|z|^2)^{n-k}= O(e^{-e^m}).$$
Hence the desired assertion for $k<n$ follows.
	\endproof
	
 Let $g:\mathbb{B}\rightarrow U$ be  a biholomorphic mapping from $\B$ to an open subset $U$ of $X$.
  Let $\psi\in C_0^{\infty}(\mathbb{B})$ such that $0\leq\psi\leq 1$ and $\psi|_{\mathbb{B}_{1/2}}=1$. Denote
$$
\tilde{\varphi}_m=(\varphi_m\psi)\circ g^{-1}.$$
Then there exists a constant $A\geq 1$ such that $\tilde{\varphi}_m$ is $A\omega$-psh for every $m>0$. Now, for all $m>A^{-n}$, we define
$$u_m=\frac{\tilde{\varphi}_m}{\sqrt[n]{m+1}}\quad\mbox{and}\quad v_m=\frac{\tilde{\varphi}_{m+1}}{\sqrt[n]{m}}.$$
We have $u_m, v_m\in \PSH(X, \omega)\cap L^{\infty}(X)$ with $\sup_X u_m= \sup_X v_m =0$ and 
$0\geq u_m\geq v_m\stackrel{L^1}{\longrightarrow}0$ as $m\rightarrow\infty$. 

Put $\mu_m:=(dd^c u_m+\omega)^n$ and $\nu_m=(dd^c v_m+\omega)^n$. We have
\begin{equation}\label{2example}
\mathbf{1}_{X\setminus g (F_m)}\mu_m+  \mathbf{1}_{X\setminus g (F_{m+1})}\nu_m\leq C_1\omega^n,
\end{equation}
for every $m$, where $C_1>0$ is a constant. By \eqref{1example}, we also have
	\begin{equation}\label{3example}
\mu_m(g(F_m))=\frac{c}{m+1}+O(e^{-e^m})\quad\mbox{and}\quad  \nu_m(g(F_{m+1}))=\frac{c}{m}+O(e^{-e^m}).
\end{equation}
By \eqref{2example}, \eqref{3example} and by the fact $v_m\stackrel{L^1}{\longrightarrow}0$, there exists $C_2>0$ such that
$$E_{\chi}(v_m)\leq C_1\int_{X\setminus g(F_{m+1})}-\chi(v_m)\omega^n
-\chi\left(\frac{-e^{m+1}}{\sqrt[n]{m}}\right)\left(\frac{c}{m}+O(e^{-e^m})\right)\leq C_2,$$
for every $m\gg 1$. Hence, $\sup_m E_{\chi}(v_m)<\infty$. Since $v_m\leq u_m\leq 0$, we  also have $\sup_m E_{\chi}(u_m)<\infty$. 
On the other hand, 
\begin{align*}
\int_{X}-\chi (v_m-u_m)(dd^c v_m+\omega)^n & \geq \int_{g(F_{m+1})}-\chi (v_m-u_m)(dd^c v_m+\omega)^n\\
&\geq \frac{c}{m}\log\left(\frac{e^{m+1}}{\sqrt[n]{m}}-\frac{e^{m}}{\sqrt[n]{m+1}}+1\right)\\
&\geq  \frac{c}{m}\log\left(\frac{(e-1)e^{m}}{\sqrt[n]{m+1}}\right)
\geq\frac{c}{2},\\
\end{align*}
for $m\gg 1$. 
It remains to show that $\lim_{m\to\infty}I_{\chi}(u_m, v_m)=0$.
By \eqref{2example} and \eqref{3example}, we have
\begin{align*}
I_{\chi}(u_m, v_m)
&=\int_{X\setminus g(F_m\cup F_{m+1})}-\chi(v_m-u_m)(\nu_m-\mu_m)\\
	&-\chi\big((v_m-u_m)(e^{-e^{m+1}})\big)\nu_m(g(F_{m+1}))+\chi\big((v_m-u_m)(e^{-e^{m}})\big)\mu_m(g(F_{m}))\\
	&\leq C_1\int_{X}|v_m|\omega^n+\frac{c}{m}\log\left(\frac{e^{m+1}}{\sqrt[n]{m}}-\frac{e^{m}}{\sqrt[n]{m+1}}+1\right)\\
		&-\frac{c}{m+1}\log\left(\frac{e^{m}}{\sqrt[n]{m}}-\frac{e^{m}}{\sqrt[n]{m+1}}+1\right)+O(e^{-e^m})\\
			&\leq C_1\int_{X}|v_m|\omega^n+\frac{c}{m}\log\left(\frac{e}{\sqrt[n]{m}}-\frac{1}{\sqrt[n]{m+1}}+e^{-m}\right)\\
			&-\frac{c}{m+1}\log\left(\frac{1}{\sqrt[n]{m}}-\frac{1}{\sqrt[n]{m+1}}+e^{-m}\right)+\frac{c}{m+1}+O(e^{-e^m})\\
			&\stackrel{m\to\infty}{\longrightarrow}0.
\end{align*}
Hence, we get $\lim_{m\to\infty} I_{\chi}(u_m, v_m)=0.$

\section{Applications} \label{sec-appli}

\subsection{Quantitative version of Dinew's uniqueness theorem}

For every Borel set $E$ in $X$,  recall that the capacity of $E$ is given by 
 $$\capK(E)=\capK_\omega(E)=\sup_{\{w \in \PSH(X,\omega): 0 \le w  \le 1\}} \int_{E}\omega_w^n.$$
 We usually remove the subscript $\omega$ from $\capK_\omega$ if $\omega$ is clear from the context. There are generalizations of capacity in big cohomology classes, many of them are comparable; see  Theorem \ref{the comparisoncap} below and \cite{Lu-comparison-capacity}. Recall that a sequence of Borel functions $(u_j)_j$ is said to \emph{converge to  a  Borel function $u$ in capacity} if  for every constant $\epsilon>0$, we have that $\capK(\{|u_j-u| \ge \epsilon\})$  converges to $0$ as $j \to \infty$. Recall that for $u_j,u \in \PSH(X, \omega)$, if $u_j \to u$ in capacity, then $u_j \to u$ in $L^1$.
 
The convergence in capacity is of great importance in pluripotential theory in part because it implies the convergence of Monge-Amp\`ere operators under reasonable circumstances. To study quantitatively the convergence in capacity, it is convenient to introduce the following distance function on $\PSH(X,\omega)$: 
$$d_{\capK}(u, v):=\sup_{\{w \in \PSH(X,\omega): 0\le w \le 1\}}\int_X |u-v|^{1/2} \omega_w^n$$
for every $u, v \in \PSH(X,\omega)$ (note that $d_{\capK}(u, v)<\infty$ thanks to the Chern-Levine-Nirenberg inequality). The number $\frac{1}{2}$ in the definition of $d_\capK$ can be replaced by any constant in $(0,1)$.  One can see that for $u_j,u \in \PSH(X,\omega)$ for $j \in \N$,   $d_{\capK}(u_j, u) \to 0$ if and only if $|u_j-u| \to 0$ in capacity. Indeed,  if $d_{\capK}(u_j, u) \to 0$, then it is clear that $|u_j-u| \to 0$ in capacity. For the converse statement, assume that $|u_j-u|$ converges to $0$ in capacity, \emph{i.e,} for every constant $\delta>0$, we have
$$\lim_{j\to\infty} \capK\big(\{|u_j-u| \ge \delta\}\big)=0.$$
In particular, the $L^1$-norm of $u_j$ is bounded uniformly in $j$. Consequently
\begin{align*}
\int_X |u_j-u|^{1/2} \omega_w^n &\le \int_{\{|u_j-u| \le \delta\}} |u_j-u|^{1/2} \omega_w^n+\int_{\{|u_j-u| \ge \delta\}} |u_j-u|^{1/2} \omega_w^n\\
& \le \delta^{1/2} \int_X \omega^n+  \bigg(\int_{\{|u_j-u| \ge \delta\}}  \omega_w^n\bigg)^{1/2} \bigg(\int_{\{|u_j-u| \ge \delta\}} |u_j-u| \omega_w^n \bigg)^{1/2} \\
&(\text{H\"older's inequality})\\
& \lesssim \delta^{1/2} \int_X \omega^n+ \bigg(\capK\big(\{|u_j-u| \ge \delta\}\big)\bigg)^{1/2},
\end{align*}
by Chern-Levine-Nirenberg inequality. Hence $d_{\capK}(u_j,u) \to 0$ if $|u_j-u|\to 0$ in capacity. The following result is an immediate consequence
of the Chern-Levine-Nirenberg inequality.

\begin{proposition}\label{pro-dominatedcapacitybigomega} Let $\theta\leq A\omega$ be a closed smooth real $(1,1)$-form (where $A\geq 1$ is a constant) and $\phi$ be a model $\theta$-psh function with $\varrho:=\int_X\theta_{\phi}^n>0$.
	Let $0 \le w \le 1$ is an $\omega$-psh function and $\psi$ is the unique solution to the problem
	\begin{equation}\label{eq propfordomination}
		\begin{cases}
			u\in\mathcal{E}(X, \theta, \phi),\\
			\theta_u^n=\dfrac{\varrho}{\vol (X)}(\ddc w+ \omega)^n,\\
			\sup_X u=0.
		\end{cases}
	\end{equation}
	Then there exists a constant $C>0$ depending only on $X$ and $\omega$ such that 
	$$\int_X|\psi|\theta_{\psi}^n\leq C\, A\varrho.$$
\end{proposition}

Here is the main result of this subsection.

\begin{theorem}\label{thequantitativeuniqueness2}
		Let $\theta\leq A\omega$ be a closed smooth real $(1,1)$-form ($A\geq 1$) 
		and let $\phi$ be a  model $\theta$-psh function such that $\varrho:=\vol(\theta_\phi)>0$. 
	Let $B \ge 1$, $\tilde{\chi} \in \widetilde{\mathcal{W}}^-$ and $u_1, u_2\in \mathcal{E}(X, \theta, \phi)$ such that $\tilde{\chi}(-1)=-1$ and
	$$E^0_{\tilde{\chi}, \theta, \phi}(u_1)+E^0_{\tilde{\chi}, \theta, \phi}(u_2) \le B.$$
	Denote $\chi(t)=\max\{t, -1\}$.
	Then, for every $0<\gamma<1$,
	there exists $C>0$ depending only on $n, X, \omega$ and $\gamma$  such that 
	\begin{equation}\label{eq0uniqueness2}
		d_{\capK}(u_1, u_2)^2\leq C\, (A+|a_1-a_2|) \left( |a_1-a_2|
		+A(A+B)^2 \lambda^{\gamma}\right),
	\end{equation}
	where $a_j:= \sup_Xu_j$,
	$\lambda=\dfrac{1}{h^{\circ n}(1/I_{\chi}^0(u_1, u_2))}$
	and $h(s)=(-\tilde{\chi}(-s))^{1/2}$.
\end{theorem}

One sees that for $u_1,u_2 \in \mathcal{E}(X, \theta, \phi)$, we can find a common $\tilde{\chi} \in \mathcal{W}^-$ so that the assumption in Theorem \ref{thequantitativeuniqueness2} is satisfied. Thus if $\sup_X u_1= \sup_X u_2= 0$, and $\theta_{u_1}^n= \theta_{u_2}^n$, then the right-hand side of (\ref{eq0uniqueness2}) vanishes, hence, $u_1= u_2$. We then recover Dinew's uniqueness theorem for prescribed singularities potentials (\cite{BEGZ,Dinew-uniqueness,Lu-Darvas-DiNezza-mono}).

\begin{proof}
	Suppose that $w$ is an arbitrary $\omega$-psh function satisfying $0\leq w\leq 1$ and 
	$\psi$ is the unique solution
	to  the problem
	\begin{equation}\label{eqpsithedinewunique}
		\begin{cases}
			u\in\mathcal{E}(X, \theta, \phi),\\
			\theta_u^n=\dfrac{\varrho}{\vol (X)}(\ddc w+ \omega)^n,\\
			\sup_X u=0.
		\end{cases}
	\end{equation}
We split the proof into two cases.\\

\noindent
\textbf{Case 1.} Assume now that $I^0_\chi(u_1,u_2) \le 1$.  \\

\noindent
Hence, we get  $\lambda=Q_{\chi, \tilde{\chi}}^{\circ n}(I^0_\chi(u_1,u_2))$ (see Remark \ref{re-nhohontru1}),  and one has $-\tilde{\chi}(-A)\leq A$ because $\tilde{\chi}(-1)=-1$.
	It follows from Theorem \ref{th-lowerenergy-kocochuanhoa} and Proposition \ref{pro-dominatedcapacitybigomega} that, 
	for every $0<\gamma<1$,
	there exists $C_1>0$ depending only on $n, X, \omega$ and $\gamma$  such that 
	\begin{equation}\label{eq1thedcap}
		I:=\int_X  -\chi\big(-|u_1- u_2|\big) \theta_\psi^n \le -\varrho \chi\left(-|a_1-a_2|-\lambda\right)
		+C_1\varrho A(A+B)^2\lambda^{\gamma}.
	\end{equation}
	Moreover
	\begin{align*}
		\dfrac{\varrho}{\vol (X)}\int_X|u_1-u_2|^{1/2}(\ddc w+\omega)^n
		&=\int_X|u_1-u_2|^{1/2}\theta_{\psi}^n\\
		&=\int_{\{|u_1-u_2|\leq 1\}}|u_1-u_2|^{1/2}\theta_{\psi}^n+\int_{\{|u_1-u_2|> 1\}}|u_1-u_2|^{1/2}\theta_{\psi}^n
		\end{align*}
		which is less than or equal to
		$$\leq I^{1/2}\left( \left(\int_{\{|u_1-u_2|\leq 1\}}\theta_{\psi}^n\right)^{1/2}+
		\left(\int_{\{|u_1-u_2|> 1\}}|u_1-u_2|\theta_{\psi}^n\right)^{1/2}\right),$$
	where the last estimate holds due to the Cauchy-Schwarz inequality. Moreover, it follows from 
	Chern-Levine-Nirenberg inequality (\cite{Kolodziej05}) that
	\begin{align} \label{ine-revisedu1u2a1a2}
		\int_{X}|u_1-a_1-u_2+a_2|\theta_{\psi}^n
		&=\dfrac{\varrho}{\vol (X)}\int_{X}|u_1-a_1-u_2+a_2|(\ddc w+ \omega)^n\\
		\nonumber
		&\leq C_2\varrho (\|u_1-a_1\|_{L^1(X)}+\|u_2-a_2\|_{L^1(X)})\\
		\nonumber
		&\leq \varrho C_3A,
	\end{align}
where $C_2, C_3>0$ depend only on $X$ and $\omega$. Here, the last estimate holds due to the 
compactness of $\{u\in\PSH(X, \omega): \sup_Xu=0\}$ in $L^1(X)$.

Hence, we have
	\begin{equation}\label{eq2thedcap}
		\dfrac{\varrho}{\vol (X)}\int_X|u_1-u_2|^{1/2}(\ddc w+\omega)^n\leq C_4 I^{1/2}\varrho^{1/2}(A+|a_1-a_2|)^{1/2},
	\end{equation}
	where $C_4>0$ depends only on $X$ and $\omega$.
	
	Combining \eqref{eq1thedcap} and \eqref{eq2thedcap}, we get
	\begin{align*}
		\left(\int_X|u_1-u_2|^{1/2}(\ddc w+\omega)^n\right)^2
		&\leq C_5(A+|a_1-a_2|)\left(-\chi\left(-|a_1-a_2|-\lambda\right)
		+A(A+B)^2\lambda^{\gamma}\right)\\
		&\leq  C_5(A+|a_1-a_2|)\left(|a_1-a_2|+\lambda
		+A(A+B)^2\lambda^{\gamma}\right)\\
		&\leq C_6(A+|a_1-a_2|)\left(|a_1-a_2|
		+A(A+B)^2\lambda^{\gamma}\right),
	\end{align*}
	where $C_5, C_6>0$ depend only on $n, X, \omega$ and $\gamma$. Since $w$ is arbitrary, we obtain desired inequality.	\\
	
\noindent
\textbf{Case 2.} We treat now the case where $I^0_\chi(u_1,u_2) \ge 1$.  
\\

\noindent
Observe that $\lambda \ge 1$ in this case. Hence the right-hand side of 
(\ref{eq0uniqueness2}) is greater than or equal to $C(A+ |a_1-a_2|)$ because $A \ge 1$ and $\lambda \ge 1$. On the other hand, H\"older's inequality gives
\begin{align*}
	\bigg(\int_X|u_1-u_2|^{1/2}(\ddc w+\omega)^n\bigg)^{2}  &\lesssim \int_X|u_1-u_2|(\ddc w+\omega)^n\\
	& \le \int_X|u_1-a_1- u_2+a_2|(\ddc w+\omega)^n+ |a_1-a_2| \int_X \omega^n\\
& \lesssim A+|a_1-a_2|
\end{align*}
by (\ref{ine-revisedu1u2a1a2}). Thus the desired estimate holds.	
		The proof is completed.
\end{proof}

\begin{remark}\label{rmkAB}
	If $B\geq A$ then the inequality \eqref{eq0uniqueness2} is equivalent to
		$$d_{\capK}(u_1, u_2)^2\leq \widetilde{C}\, (A+|a_1-a_2|) \left( |a_1-a_2|
		+A\, B^2 \lambda^{\gamma}\right),$$
	where $\tilde{C}>0$ depends only on $n, X, \omega$ and $\gamma$.
\end{remark}

\subsection{Quantitative version for the domination principle}

\begin{theorem}\label{the domination}
	Let $A \ge 1$ be a constant and let $\theta\leq A\omega$ be a closed smooth real $(1,1)$-form 
	 and $\phi$ be a  model $\theta$-psh function, and $\varrho:=\vol(\theta_\phi)>0$.  Let $B \ge 1$ be a constant, $\tilde{\chi} \in \widetilde{\mathcal{W}}^-$ and $u_1, u_2\in \mathcal{E}(X, \theta, \phi)$ such that $\tilde{\chi}(-1)=-1$ and
	$$E^0_{\tilde{\chi}, \theta, \phi}(u_1)+E^0_{\tilde{\chi}, \theta, \phi}(u_2) \le B.$$
	Assume that there exists a constant $0\leq c<1$ and a Radon measure $\mu$ on $X$ satisfying
	$\theta_{u_1}^n\leq c\theta_{u_2}^n+\varrho\mu$ on $\{u_1<u_2\}$ and 
	$c_{\mu}:=\int_{\{u_1<u_2\}}d\mu\leq 1$.
	Then there exists a constant $C>0$ depending only on $n, X$ and $\omega$   such that 
	$$\capK_{\omega}\{u_1<u_2-\epsilon\}\leq 
	\dfrac{C\vol(X) (A+B)^2}{\epsilon(1-c)h^{\circ n}(1/c_{\mu})},$$
	for every $0<\epsilon<1$,
	where  $h(s)=(-\tilde{\chi}(-s))^{1/2}$ for every $0 \le s \le \infty$.
	
	In particular, if $c_{\mu}=0$ then $\capK_{\omega}\{u_1<u_2-\epsilon\}=0$ for every $\epsilon>0$,
	and then $u_1\geq u_2$ on whole $X$.
\end{theorem}

The standard domination principle corresponds to the case where $c=0$ and $\mu:=0$. A non-quantitative version of this domination principle (\emph{i.e,} for $\mu=0$) in the non-K\"ahler setting was obtained in \cite{Guedj-Lu-3}.

\begin{proof}[Proof of Theorem \ref{the domination}]
	Let $w$ be an arbitrary $\omega$-psh function satisfying $0\leq w\leq 1$ and 
	$\psi$ is the unique solution
	to \eqref{eq propfordomination}. Denote $v=\max\{u_1, u_2\}$ and 
	$\chi(t)=\max\{t, -1\}\geq \tilde{\chi}(t)$. By Theorem \ref{th-lowerenergy}
	and Proposition \ref{pro-dominatedcapacitybigomega}, there exists  a constant $C_1>0$ depending only on $n, X$ an 
	$\omega$ such that
	\begin{equation}\label{eq1 proof domination}
		I_1:=\int_X  -\chi (u_1-v)  (\theta_{\psi}^n- \theta_{u_1}^n) \le 
		C_1\varrho (A+B)^2 Q^{\circ (n)}(I^0_\chi(u_1, v)),
	\end{equation}
	and
	\begin{equation}\label{eq2 proof domination}
		I_2:=\int_X  -\chi (u_1-v)  (\theta_{u_2}^n- \theta_{u_1}^n) \le 
		C_1\varrho (A+B)^2 Q^{\circ (n)}(I^0_\chi(u_1, v)).
	\end{equation}
	Moreover,  by the fact $\theta_v^n=\theta_{u_2}^n$ on $\{u_1<u_2\}$ and
	by the assumption $\theta_{u_1}^n\leq c\theta_{u_2}^n+\varrho\mu$ on $\{u_1<u_2\}$
	, we have
	\begin{equation}\label{eq3 proof domination}
		I^0_\chi(u_1, v)
		=\varrho^{-1}\int_{\{u_1<u_2\}} -\chi (u_1-v)(\theta_{u_1}^n-\theta_{u_2}^n)
		\leq\varrho^{-1}\int_{\{u_1<u_2\}} -\chi (u_1-v)(\theta_{u_1}^n-c\theta_{u_2}^n)\leq c_{\mu}.
	\end{equation}
	Combining \eqref{eq1 proof domination}, \eqref{eq2 proof domination} and \eqref{eq3 proof domination}, we get
	\begin{align*}
		(1-c)	\int_X  -\chi (u_1-v)  \theta_{\psi}^n
		&=\int_X -\chi (u_1-v)(\theta_{u_1}^n-c\theta_{u_2}^n)+(1-c)I_1+cI_2\\
		&\leq \int_X -\chi (u_1-v)(\theta_{u_1}^n-c\theta_{u_2}^n)
		+ C_1\varrho (A+B)^2 Q^{\circ n}(c_{\mu})\\
		&\leq \varrho c_{\mu}+ C_1\varrho (A+B)^2 Q^{\circ n}(c_{\mu})\\
		&\leq C\varrho (A+B)^2 Q^{\circ n}(c_{\mu}),
	\end{align*}
	where $C=C_1+1$. Hence
	$$\int_{\{u_1<u_2-\epsilon\}}\omega_{w}^n=\dfrac{\vol (X)}{\varrho}\int_{\{u_1<u_2-\epsilon\}}\theta_{\psi}^n
	\leq \dfrac{C\vol(X) (A+B)^2 Q^{\circ n}(c_{\mu})}{(1-c)\epsilon},$$
	for every $0<\epsilon<1$. Since $w$ is arbitrary, it follows that
	\begin{equation}\label{eq4 proof domination}
		\capK_{\omega}\{u_1<u_2-\epsilon\}\leq 
		\dfrac{C\vol(X) (A+B)^2 Q^{\circ n}(c_{\mu})}{(1-c)\epsilon}.
	\end{equation}
	Moreover, by the definition of $\chi$ and the formula of $Q$, we have
	$$Q(s)=\dfrac{1}{(-\tilde{\chi}(-1/s))^{1/2}}=\dfrac{1}{h(1/s)},$$
	for every $0<s\leq 1$, and $Q(0)=0$. Then
	\begin{equation}\label{eq5 proof domination}
		Q^{\circ n}(s)=\dfrac{1}{h^{\circ n}(1/s)},
	\end{equation}
for every $0\leq s\leq 1$.
	The proof is completed.
\end{proof}

\subsection{Relation to Darvas's metrics on the space of potentials of finite energy} \label{subsec-metricenergy}

Let $\chi \in \mathcal{W}^- \cup \mathcal{W}^+_M$. Let $\theta$ be a closed smooth real $(1,1)$-form in a big cohomology class. When $\theta$ is K\"ahler, it was proved in \cite{Darvas-finite-energy,Darvas-kahlerclass,Darvas-lower-energy} that there is a natural metric $d_\chi$ on $\mathcal{E}_\chi(X, \theta)$ which makes the last space to be a complete metric space. When $\chi(t)= t$, such metrics have a long history and play an important role in the study of complex Monge-Amp\`ere equations. We refer to these last references and  \cite{Berman-Boucksom-Jonsson-KE,Berman-Darvas-Lu} for more details.  
We now draw the connection between $I_\chi(u,v)$ and the metric on $\mathcal{E}_\chi(X,\theta)$. 
Let 
$$\tilde{I}_\chi(u,v)= \int_{\{u<v\}} -\chi(u-v) (\theta_v^n + \theta_u^n)+\int_{\{u>v\}} -\chi(v-u) (\theta_u^n + \theta_v^n) \ge I_\chi(u,v).$$
By \cite{Darvas-finite-energy,Darvas-kahlerclass,Darvas-lower-energy}, there exists a constant $C>0$ such that 
$$C^{-1} \tilde{I}_\chi(u,v) \le d_\chi(u,v) \le C \tilde{I}_\chi(u,v)$$
for every $u,v \in \mathcal{E}_\chi(X,\theta)$ and $\theta$ is K\"ahler. It was proved in \cite{Gupta} (and also \cite{Darvas-finite-energy,DDL-L1metric,Lu-DiNezza-Lpmetric,Trusiani-energy,Xia-energy}) that   $\tilde{I}_\chi(u,v)$ satisfies a quasi-triangle inequality, and  the convergence in  $\tilde{I}_\chi(u,v)$  implies the convergence in capacity by using the plurisubharmonic envelope. Such a method is not quantitative.  We present below quantitative version of this fact by using our approach. 


\begin{theorem}\label{the-capmetricdarvas1}
	Let $\theta\leq A\omega$ be a closed smooth real $(1,1)$-form ($A\geq 1$ is a constant) 
	and $\phi$ be a  model $\theta$-psh function with $\varrho:=\vol(\theta_\phi)>0$.  
		Let $B\geq 1$, $\tilde{\chi} \in \mathcal{W}^-$ and $u_1, u_2\in \mathcal{E}(X, \theta, \phi)$ such that 
		 $\tilde{\chi}(-1)=-1$ and
	$$E^0_{\tilde{\chi}, \theta, \phi}(u_1)+E^0_{\tilde{\chi}, \theta, \phi}(u_2) \le B.$$
	Then there exist  $C>0$ depending only on $n, X$ and $\omega$ such that
	$$d_{\capK}(u_1, u_2)^2\leq  \dfrac{C\,\big(A+|\sup_Xu_1-\sup_Xu_2|\big)\, (A+B)^2}{h^{\circ n}(\varrho/\tilde{I}_{\tilde{\chi}}(u_1, u_2))},$$
	where $h(s)=(-\tilde{\chi}(-s))^{1/2}$ for every $0 \le s \le \infty$.
\end{theorem}

We note that the quantities $a_j:= |\sup_Xu_j|$ for $j=1,2$ (hence $|a_1-a_2|$) can be bounded by a function of $B$ and $\tilde{\chi}$ as follows. Since $\phi$ is a model, we have $-a_j= \sup_X (u_j- \phi)$. It follows that 
  $$B \ge E^0_{\tilde{\chi}, \theta, \phi}(u_j) \ge -\tilde{\chi}(-a_j).$$
Consequently, we get $a_j \le -\tilde{\chi}^{-1}(-B)$  for $j=1,2$, where $\tilde{\chi}^{-1}$ denotes the inverse map of $\tilde{\chi}: \R_{\le 0} \to \R_{\le 0}$.
Thus by Theorem \ref{the-capmetricdarvas1}, one sees that if $\tilde{I}_{\tilde{\chi}}(u_1, u_2)$ is small, then so is $d_\capK(u_1,u_2)$ (uniformly in $u_1,u_2 \in \mathcal{E}(X, \theta,\phi)$ of $\tilde{\chi}$-energy  bounded by a fixed constant). 

\begin{proof}
	Let $\chi(t)=\max\{t, -1\}$.
	Suppose that $w$ is an arbitrary $\omega$-psh function satisfying $0\leq w\leq 1$. By the proof of Theorem \ref{thequantitativeuniqueness2} (see \eqref{eq2thedcap}), there exists $C_1>0$ depending only on $X$ and $\omega$ such that
\begin{equation}\label{eq1darvas1}
\left(\int_X|u_1-u_2|^{1/2}(\ddc w+\omega)^n\right)^2\leq C_1\big(A+|\sup_Xu_1-\sup_Xu_2|\big) \varrho^{-1}\int_X  -\chi\big(-|u_1- u_2|\big) \theta_{\psi}^n,
\end{equation}
where $\psi$ is defined by \eqref{eqpsithedinewunique}. Moreover, it follows from Theorem \ref{th-lowerenergy} 
(applied to $u_1, \max\{u_1,u_2\},$ $\psi_1:=\psi, \psi_2:=u_1$)
and Proposition \ref{pro-dominatedcapacitybigomega} that
$$\int_X  -\chi\big(-|u_1- u_2|\big) \theta_{\psi}^n\leq  \tilde{I}_{\chi}(u_1, u_2)
+C_2\varrho (A+B)^2Q_{\chi, \tilde{\chi}}^{\circ (n)}(I^0_\chi(u_1,u_2)),$$
where $C_2>0$ depends only on $n$. Therefore, by the facts 	$Q^{\circ(n)}(s)=\dfrac{1}{h^{\circ(n)}(1/s)}$ and 
$I_\chi(u_1,u_2)\leq \tilde{I}_{\chi}(u_1, u_2)\leq \tilde{I}_{\tilde{\chi}}(u_1, u_2))$, we obtain
\begin{equation}\label{eq2darvas1}
\int_X  -\chi\big(-|u_1- u_2|\big) \theta_{\psi}^n\leq
\dfrac{C_3\varrho (A+B)^2}{h^{\circ(n)}(\varrho/\tilde{I}_{\tilde{\chi}}(u_1, u_2))},
\end{equation}
	where $C_3>0$ depends only on $n, X$ and $\omega$.	Combining \eqref{eq1darvas1} and \eqref{eq2darvas1}, we get 
	$$\left(\int_X|u_1-u_2|^{1/2}(\ddc w+ \omega)^n\right)^2\leq \dfrac{C\, \big(A+|\sup_Xu_1-\sup_Xu_2|\big)\, (A+B)^2}{h^{\circ(n)}(\varrho/\tilde{I}_{\tilde{\chi}}(u_1, u_2))},$$
	where $C>0$ depends only on $n, X$ and $\omega$. Since $w$ is arbitrary, we get the desired inequality.
		The proof is completed.
\end{proof}

\begin{remark}\label{rebosungWcongM} Consider now a weight $\tilde{\chi} \in \mathcal{W}^+_M$ with $\tilde{\chi}(-1)=-1$. One sees that $\tilde{\chi}(t) \le (-t)^{M} \tilde{\chi}(-1)= -(-t)^M$ for $-1 \le t \le 0$, and  $\tilde{\chi}(t) \le \tilde{\chi}_0(t):=t$ for $t \le -1$. Consequently, using H\"older's inequality, we obtain 
$$\rho^{-1}\tilde{I}_{\tilde{\chi}_0}(u_1,u_2) \le  2\big(\rho^{-1}\tilde{I}_{\tilde{\chi}}(u_1,u_2)\big)^{1/M}+ \rho^{-1}\tilde{I}_{\tilde{\chi}}(u_1,u_2).$$
Hence,  Theorem \ref{the-capmetricdarvas1} applied to $\tilde{\chi}_0$ shows that if $\rho^{-1}\tilde{I}_{\tilde{\chi}}(u_1,u_2) \to 0$ and the normalized $\tilde{\chi}$-energies of $u_1,u_2$ are uniformly bounded, then $d_{\capK}(u_1,u_2) \to 0$. 
\end{remark}

When $\tilde{\chi} \in \mathcal{W}^+_M$, we have another version of Theorem \ref{the-capmetricdarvas1} which is more explicit.

\begin{theorem}\label{the-capmetricdarvas2}
	Let $\theta\leq A\omega$ be a closed smooth real $(1,1)$-form ($A\geq 1$) and $\phi$ be a  model $\theta$-psh function such that $\varrho:=\vol(\theta_\phi)>0$.  Let $B\ge 1$,  $\tilde{\chi} \in \mathcal{W}_M^+$ ($M\geq 1$)  and $u_1, u_2\in \mathcal{E}(X, \theta, \phi)$ such that $\tilde{\chi}(-1)=-1$ and
	$$E^0_{\tilde{\chi}, \theta, \phi}(u_1)+E^0_{\tilde{\chi}, \theta, \phi}(u_2) \le B.$$
	Then there exists  $C>0$ depending only on $n$ and $M$ such that
	\begin{align} \label{ine-chingaWcongMrevise}
	\int_X-\tilde{\chi}(-|u_1-u_2|)\theta_{\psi}^n
		\leq C\varrho B^2\left(\tilde{I}_{\tilde{\chi}}(u_1, u_2)/\varrho\right)^{2^{-n}},
	\end{align}
for every $\psi\in\PSH (X, \theta)$ with  $\phi-1\leq \psi\leq\phi$. Moreover, if $\sup_X u_1=\sup_X u_2$ then 
there exists $C'>0$ depending on $n, X, \omega, A$ and $M$ such that
	\begin{align*}
	\tilde{I}_{\tilde{\chi}}(u_1, u_2)
	\le 
	C'\varrho A^{1/2}B^2\left(I^0_{\tilde{\chi}}(u_1, u_2)\right)^{2^{-n-1}}.
\end{align*}
\end{theorem}

\begin{proof}
The case $I^0_{\tilde{\chi}}(u_1, u_2)\geq 1$ is trivial because we have 
$$\tilde{I}_{\tilde{\chi}}(u_1, u_2)/\varrho \ge I^0_{\tilde{\chi}}(u_1, u_2)\geq 1,$$
whereas the left-hand side of (\ref{ine-chingaWcongMrevise}) is always bounded by a constant (depending on $M$) times $B$. Thus, from now on, it suffices to assume that 
	$I^0_{\tilde{\chi}}(u_1, u_2)<1$.
	
Denote $v=\max\{u_1, u_2\}$.
By Lemma \ref{le-sosanhnangnluongintegrabig}, we have 
$v\in \mathcal{E}(X, \theta, \phi)$ 
and $E^0_{\tilde{\chi}, \theta, \phi}(v)\leq C_1B,$ where $C_1>0$ depends only on $n$ and $M$. 
Taking $\chi=\tilde{\chi}$ and using Theorem \ref{th-lowerenergy}, we get
\begin{equation}\label{eq1 th2darvas}
		\int_X-\tilde{\chi}(u_j-v)\theta_{\psi}^n\leq \int_X-\tilde{\chi}(u_j-v)\theta_{u_j}^n
		+C_2\varrho B^2 \big(I^0_{\tilde{\chi}}(u_j, v)\big)^{2^{-n}},
\end{equation} 
for $j=1,2$, where $C_2>0$ depends on $n$ and $M$. Note that
$$\int_X-\tilde{\chi}(u_1-v)\theta_{u_1}^n+\int_X-\tilde{\chi}(u_2-v)\theta_{u_2}^n
\leq\int_X-\tilde{\chi}(-|u_1-u_2|)(\theta_{u_1}^n+\theta_{u_2}^n)=\tilde{I}_{\tilde{\chi}}(u_1, u_2),$$
and
$$I^0_{\tilde{\chi}}(u_1, v)+I^0_{\tilde{\chi}}(u_2, v)=I^0_{\tilde{\chi}}(u_1, u_2)\leq\varrho^{-1}\tilde{I}_{\tilde{\chi}}(u_1, u_2).$$
Hence, by \eqref{eq1 th2darvas}, we get
\begin{align*}
	\int_X-\tilde{\chi}(-|u_1-u_2|)\theta_{\psi}^n
	&=\int_X-\tilde{\chi}(u_1-v)\theta_{\psi}^n+\int_X-\tilde{\chi}(u_2-v)\theta_{\psi}^n\\
	&\leq \int_X-\tilde{\chi}(u_1-v)\theta_{u_1}^n+\int_X-\tilde{\chi}(u_2-v)\theta_{u_2}^n\\
	&+C_2\varrho B^2 \left(\big(I^0_{\tilde{\chi}}(u_1, v)\big)^{2^{-n}}+
	\big(I^0_{\tilde{\chi}}(u_2, v)\big)^{2^{-n}}\right)\\
	&\leq \tilde{I}_{\tilde{\chi}}(u_1, u_2)
	+2C_2\varrho B^2 \big(\tilde{I}_{\tilde{\chi}}(u_1, u_2)/\varrho\big)^{2^{-n}}\\
	&\leq C_3\varrho B^2 \big(\tilde{I}_{\tilde{\chi}}(u_1, u_2)/\varrho\big)^{2^{-n}},
\end{align*}
where $C_3>0$ depends on $n$ and $M$. Here, the last estimate holds due to the fact
$\tilde{I}_{\tilde{\chi}}(u_1, u_2)\leq \varrho B$.

Now, we consider the case $\sup_Xu_1=\sup_Xu_2$. By Theorem \ref{th-lowerenergy-kocochuanhoa}
(choose $m=1$ and $\gamma=1/2$),
there exists $C_4>0$ depending only on $n, X, \omega$ and $M$ such that
\begin{align}\label{eq2 th2darvas}
\tilde{I}_{\tilde{\chi}}(u_1, u_2) &\leq	\int_X-\tilde{\chi}(-|u_1-u_2|)(\theta_{u_1}^n+\theta_{u_2}^n)\\
\nonumber
&\leq
	-2\varrho\,\tilde{\chi}\left(-\big(I^0_{\tilde{\chi}}(u_1, u_2)\big)^{2^{-n}}\right)
	+C_4\varrho A^{1/2} B^2 \big(I^0_{\tilde{\chi}}(u_1, u_2)\big)^{2^{-n-1}}.
\end{align}
Moreover, since $\tilde{\chi}$ is concave, we have 
$$\dfrac{\tilde{\chi}(t)}{t}\leq \dfrac{\tilde{\chi}(-1)}{-1}=1,$$
for every $-1<t<0$. Hence, by \eqref{eq2 th2darvas}, we have
\begin{align*}
	\tilde{I}_{\tilde{\chi}}(u_1, u_2)
	&\leq 2\varrho\,\big(I^0_{\tilde{\chi}}(u_1, u_2)\big)^{2^{-n}}
	+C_4\varrho A^{1/2}B^2 \big(I^0_{\tilde{\chi}}(u_1, u_2)\big)^{2^{-n-1}}\\
	&\leq (2+C_4)\varrho A^{1/2}B^2 \big(I^0_{\tilde{\chi}}(u_1, u_2)\big)^{2^{-n-1}}.
\end{align*}
The proof is completed.
\end{proof}

\subsection{Comparison of capacities}
 
For every Borel subset $E$ in $X$ and for every $\varphi \in \PSH(X,\theta)$, we recall again that 
$$\capK_{\theta,\varphi}(E)= \sup\big\{\int_E \theta_\psi^n: \, \psi \in \PSH(X, \theta), \quad \varphi- 1 \le \psi \le \varphi\big\}.$$
In \cite{Lu-comparison-capacity}, it was showed that if $\varphi_j$ ($j=1, 2$) is a  $\theta_j$-psh function with
 $\int_X(\theta_j+dd^c\varphi_j)^n>0$ then there exists a continuous function $f:\R_{ \ge 0}\rightarrow\R_{ \ge 0}$
 with $f(0)=0$ such that $\capK_{\theta_1,\varphi_1}(E)\leq f(\capK_{\theta_2,\varphi_2}(E))$ for every
 Borel set $E\subset X$.  As an application of our main results,  we obtain the following quantitative comparison of capacities  for the case where $\varphi_j$ is a model $\theta_j$-psh function.

\begin{theorem}\label{the comparisoncap} (Comparison of capacities)
Assume that $\theta_1, \theta_2\leq A\omega$ are closed smooth real $(1,1)$-forms representing big cohomology classes	and, for $j=1,2$,  $\phi_j$ is a model $\theta_j$-psh function satisfying 
	$\int_X(dd^c\phi_j+\theta_j)^n=\varrho_j>0$. Then, for every $0<\gamma<1$, there exists $C>0$
	depending only on $n, X, \omega, A$ and $\gamma$ such that
	$$\frac{\capK_{\theta_1, \phi_1}(E)}{\varrho_1}\leq 
C	\bigg(\frac{\capK_{\theta_2, \phi_2}(E)}{\varrho_2}\bigg)^{2^{-n}\gamma},$$
	for every Borel set $E\subset X$.
\end{theorem}

We now prove Theorem \ref{the comparisoncap}.  First, we need the following lemma.

\begin{lemma}\label{lem CLNfortheta}
	Let $A, B>0$ be constants.
	Let $\theta$ be a closed smooth real $(1,1)$-form representing a big cohomology class such that 
	$\theta\leq A\omega$. Assume that $u, v$ are $\theta$-psh functions satisfying $v\leq u\leq v+B$.
	Then, 
	$$\int_X(-\psi)\theta_u^n\leq\int_X(-\psi)\theta_v^n+nA^nB\int_X\omega^n,$$
	for every negative $A\omega$-psh function $\psi$.
\end{lemma}

\begin{proof}
	Using approximations, we can assume that $\psi$ is smooth. 
	Denote 
	$$T=\sum_{l=0}^{n-1}\theta_u^l\wedge\theta_v^{n-l-1}.$$
	We have $\theta_u^n-\theta_v^n=dd^c(u-v)\wedge T$.
	Moreover,  using integration by parts (Theorem \ref{th-integrabypart}), we get
	$$\int_X(-\psi)dd^c(u-v)\wedge T
	=\int_X (u-v)dd^c(-\psi)\wedge T\leq A\int_X (u-v)\omega\wedge T\leq nA^nB\int_X\omega^n.$$
	Hence
	$$\int_X(-\psi)\theta_u^n\leq\int_X(-\psi)\theta_v^n+nA^nB\int_X\omega^n.$$
\end{proof}

\begin{proof}[Proof of Theorem \ref{the comparisoncap}]
By the inner regularity of capacities (see \cite[Lemma 4.2]{Lu-Darvas-DiNezza-mono}), we only need consider
the case where $E$ is compact. Since the case $\capK_{\theta_2, \phi_2}(E)=\varrho_2$ is trivial, we can also
assume that $\capK_{\theta_2, \phi_2}(E)<\varrho_2$. In particular,
 by \cite[Proposition 3.7]{Lu-Darvas-DiNezza-logconcave} and \cite[Lemma 2.7]{Darvas-Lu-DiNezza-singularity-metric},
 we have
 $$\sup_Xh_{E, \theta_2, \phi_2}^*=\sup_X(h_{E, \theta_2, \phi_2}^*-\phi_2)=0,$$
  where 
 $$h_{E, \theta_2, \phi_2}=\sup\left\{w\in\PSH(X, \theta_2): w|_E\leq\phi_2-1, w\leq\phi_2\right\}.$$
Set $\chi(t)=\tilde{\chi}(t)=t$. We will use Theorem \ref{th-lowerenergy-kocochuanhoa} for $u_1=(h_{E, \theta_2, \phi_2})^*$ and $u_2=\phi_2$.
It is clear that $E^0_{\tilde{\chi}, \theta_2, \phi_2}(u_2)=0$ and $u_1=u_2-1$ on $E\setminus N$, where
$N$ is a pluripolar set. Moreover, it follows from
\cite[Proposition 3.7]{Lu-Darvas-DiNezza-logconcave} that
$$I^0_{\chi}(u_1, u_2) \le E^0_{\tilde{\chi}, \theta_2, \phi_2}(u_1)=\varrho_2^{-1}\capK_{\theta_2, \phi_2}(E)\leq 1.$$
 By Theorem \ref{th-lowerenergy-kocochuanhoa}, for every $0<\gamma<1$ and $B\geq 1$, there exists $C>0$
depending only on $X, \omega, n, A$ and $\gamma$ such that
\begin{equation}\label{eq2 comparisoncap}
	\int_E\theta_{\psi}^n\leq\int_X\-\chi(-|u_1-u_2|)\theta_{\psi}^n\leq
	 C\varrho_2A(A+B)^2\left(\capK_{\theta_2, \phi_2}(E)/\varrho_2\right)^{2^{-n}\gamma},
\end{equation}
for every compact set $E$  and for each 
$\psi\in\mathcal{E}(X, \theta_2, \phi_2)$ with $E^0_{\tilde{\chi}, \theta_2, \phi_2}(\psi)\leq B$.
Let $\varphi\in \mathcal{E}(X, \theta_1, \phi_1)$ such that $\phi_1-1\leq\varphi\leq\phi_1$
and $\int_E(\theta_1+dd^c\varphi)^n\geq\dfrac{1}{2}\capK_{\theta_1, \phi_1}(E)$.
By \cite{Lu-Darvas-DiNezza-logconcave}, there exists a unique function 
$\psi_0\in\mathcal{E}(X, \theta_2, \phi_2)$ such that $\sup_X\psi_0=0$ and
$$(dd^c\psi_0+\theta_2)^n=\dfrac{\varrho_2}{\varrho_1}(dd^c\varphi+\theta_1)^n.$$
When $\psi=\psi_0$, we have
\begin{equation}\label{eq3 comparisoncap}
	\int_E\theta_{\psi}^n\geq \dfrac{\varrho_2}{2\varrho_1}\capK_{\theta_1, \phi_1}(E).
\end{equation}
Moreover, by using Lemma \ref{lem CLNfortheta} for $\varphi, \phi_1$ and using the fact that 
$(dd^c\phi_2+\theta_2)^n\leq\mathbf{1}_{\{\phi_2=0\}}\theta_2^n$ (see \cite[Theorem 3.8]{Lu-Darvas-DiNezza-mono}), we have
\begin{equation}\label{eq4 comparisoncap}
	\varrho_1 E^0_{\tilde{\chi}, \theta_2, \phi_2}(\psi_0)
	= \int_X(\phi_2-\psi_0)(dd^c\varphi+\theta_1)^n
	\leq\int_X(-\psi_0)(dd^c\phi_1+\theta_1)^n+nA^n\int_X\omega^n
	\leq B,
\end{equation}
where $B\geq 1$ depends only on $A, X, \omega, n$. Combining \eqref{eq2 comparisoncap}, \eqref{eq3 comparisoncap}
and \eqref{eq4 comparisoncap}, we get
\begin{align*}
	\capK_{\theta_1, \phi_1}(E)\leq \dfrac{2\varrho_1}{\varrho_2}\int_E\theta_{\psi_0}^n&\leq
	\dfrac{2\varrho_1}{\varrho_2}\int_X\-\chi(-|u_1-u_2|)\theta_{\psi_0}^n\\
	&\leq
	2C\varrho_1A(A+B)^2\left(\capK_{\theta_2, \phi_2}(E)/\varrho_2\right)^{2^{-n}\gamma}.
\end{align*}
The proof is completed.
\end{proof}


\bibliography{biblio_family_MA,biblio_Viet_papers,bib-kahlerRicci-flow}

\bibliographystyle{siam}

\bigskip

\noindent
\Addresses
\end{document}